\documentclass[10pt,a4paper,reqno]{amsart} 

\usepackage[english]{babel}
\usepackage[utf8]{inputenc}
\usepackage{hyperref}
\usepackage{latexsym}
\usepackage{verbatim}
\usepackage{rotating}
\usepackage{color,bbm}
\usepackage{amssymb,amsfonts,amsmath,stmaryrd,dsfont}
\usepackage[T1]{fontenc}
\usepackage[foot]{amsaddr}

\let\mathbbm\mathds

\newtheorem{theorem}{Theorem}

\newtheorem{conjecture}[theorem]{Conjecture}

\newtheorem{prop}[theorem]{Proposition}

\newtheorem{definition}[theorem]{Definition}
\theoremstyle{definition}

\newtheorem{remark}[theorem]{Remark}

\newcommand{\beq}{\begin{equation}}
\newcommand{\eeq}{\end{equation}}
\def\emm#1,{{\em #1}}

\catcode`\@=11
\def\section{\@startsection{section}{1}%
 \z@{.7\linespacing\@plus\linespacing}{.5\linespacing}%
 {\normalfont\bfseries\scshape\centering}}

\def\subsection{\@startsection{subsection}{2}%
  \z@{.5\linespacing\@plus\linespacing}{.5\linespacing}%
  {\normalfont\bfseries\scshape}}

\def\subsubsection{\@startsection{subsubsection}{3}%
 \z@{.5\linespacing\@plus\linespacing}{-.5em}
 {\normalfont\bfseries}}
\catcode`\@=12

\newcommand{\cs}{\mathbb{C}}
\newcommand{\fps}{formal power series}

\newcommand{\gw}{{\rm{\bf w}}}
\newcommand{\gu}{\textbf{u}}
\newcommand{\gv}{\textbf{v}}
\newcommand{\ba}{\bar{a}}

\DeclareMathOperator{\dv}{dv}

\DeclareMathOperator{\ee}{e}

\newcommand{\gf}{generating function}
\newcommand{\gfs}{generating functions}

\renewcommand{\epsilon}{\varepsilon}
\newcommand{\vareps}{\varepsilon}
\newcommand{\la}{\lambda}

\newcommand{\pinf}{\mathbb{L}^{(k)}}
\newcommand{\PEO}{\mathcal{O}} 
\newcommand{\peom}{\mathcal{L}^{(k)}}
\newcommand{\Kc}{K} 
\newcommand{\Lc}{L} 
\newcommand{\tLc}{\tilde L} 
\newcommand{\Le}{\Lc_\varepsilon} 

\newcommand{\Kp}{{\sf K}} 
\newcommand{\Lp}{{\sf L}} 
\newcommand{\tLp}{\tilde{\sf L}} 
\newcommand{\Kpp}{{\sf K}'} 
\newcommand{\Lpp}{{\sf L}'} 
\newcommand{\Lppp}[1]{{{\sf L}'}^+_{#1}} 
\newcommand{\tLpp}{\tilde{\sf L}'} 

\newcommand{\Lpe}{\Lp_\varepsilon} 

\newcommand{\psup}{\mathbb{U}^{(k)}}
\newcommand{\peop}{\mathcal{U}^{(k)}}

\newcommand{\Tc}{T} 
\newcommand{\Ucl}{U} 
\newcommand{\Uc}[1]{U_{#1}^x} 
\newcommand{\Ucun}[1]{U_{#1}} 
\newcommand{\Upl}{{\sf U}} 
\newcommand{\Up}[1]{{\sf U}_{#1}^x}
\newcommand{\Upun}[1]{{\sf U}_{#1}} 
\newcommand{\Tp}{{\sf T}} 
\newcommand{\Uppl}{{\sf U}'} 
\newcommand{\Upp}[1]{{{\sf U}'}_{#1}^x}
\newcommand{\Uppun}[1]{{\sf U}'_{#1}} 
\newcommand{\Tpp}{{\sf T}'} 
\newcommand{\Ac}{A} 
\newcommand{\Bc}{B} 
\newcommand{\Cc}{C} 
\newcommand{\Ap}{A} 
\newcommand{\Bp}{B} 
\newcommand{\Cp}{C} 
\newcommand{\App}{A'} 
\newcommand{\Bpp}{B'} 
\newcommand{\Cpp}{C'} 

\DeclareMathOperator{\Pol}{{\rm Pol}}

\newcommand{\e}{f}

\begin{document}
\title[On the number of planar Eulerian orientations]
{On the number of planar Eulerian orientations}

\author[N. Bonichon]{Nicolas Bonichon}

\author[M. Bousquet-M\'elou]{Mireille Bousquet-M\'elou}

\author[P. Dorbec]{Paul Dorbec}

\author[C. Pennarun]{Claire Pennarun}

\thanks{{NB, MBM  and CP were partially supported by the French ``Agence Nationale
de la Recherche'',  via grants JCJC EGOS ANR-12-JS02-002-01 (NB
 and CP) and  Graal ANR-14-CE25-0014 (MBM)}}

\address{CNRS, LaBRI, Universit\'e de Bordeaux, 351 cours de la
  Lib\'eration,  F-33405 Talence Cedex, France} 
\email{bonichon,bousquet,dorbec,cpennaru@labri.fr}

\begin{abstract}
The number of  planar Eulerian maps with $n$ edges is well-known to
have a simple expression. But what is the number of planar Eulerian \emm
orientations, with $n$ edges? This problem appears to be
difficult. To approach it, we
define and count families of subsets and supersets of planar Eulerian 
orientations, indexed by an integer $k$, that converge to the set of
all planar Eulerian orientations as $k$ increases. The \gfs\ of our
subsets can be characterized by systems of  polynomial equations, and
are thus algebraic. The \gfs\ of our
supersets   are characterized by polynomial systems involving divided
differences, as often occurs in map enumeration. We prove that these series are algebraic as well. We obtain in this way
lower and upper bounds on the growth rate  of planar Eulerian
orientations, which appears to be around $12.5$.
\end{abstract}

\maketitle

\section{Introduction}

 The enumeration of planar maps (graphs embedded on the sphere) has
received a lot of attention since the sixties.  Many remarkable
counting results have been discovered, which were often illuminated
later by beautiful
bijective constructions. For instance, it has been known\footnote{in
  disguise! The 1963 result involves \emm bicubic, maps, which are in
  one-to-one correspondence with Eulerian
  maps. See e.g.~\cite[Cor.~2.4]{mbm-schaeffer-constellations} for the dual
  bijection between face-bicoloured triangulations and bipartite maps.} since 1963
that the number of rooted planar \emm Eulerian, maps (i.e., planar
maps in which every vertex has even degree) with $n$ edges is~\cite{tutte-census-maps}:
\beq\label{euler-count}
m_n= \frac{3 \cdot 2^{n-1}}{(n+1)(n+2)}{2n\choose n}.
\eeq
A bijective explanation involving plane trees can be found
in~\cite{mbm-schaeffer-constellations}. The associated \gf\ $
 M(t)=\sum_{n\ge 0} m_n t^n
$
is known to be \emm algebraic,, that
is, to satisfy a polynomial equation. More precisely:
$$
t^2+11 t-1-(8 t^2+12 t-1) M(t)+16 t^2 M(t)^2=0.
$$

\begin{figure}[h]
\centering{\includegraphics[scale=0.8]{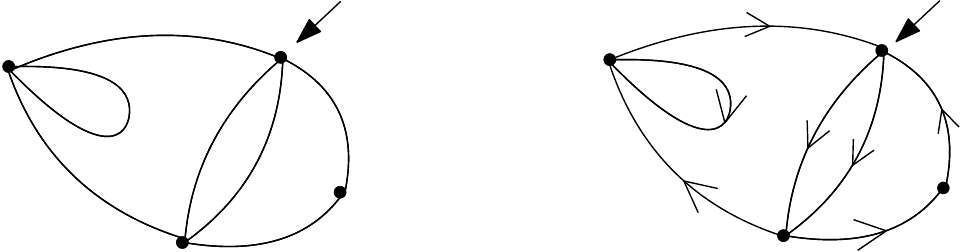}}
\caption{A rooted Eulerian map and a rooted Eulerian orientation.}
\label{fig:euler}
\end{figure}

Beyond their enumerative implications, bijections involving maps  have  been applied to 
encode, sample and draw maps efficiently~\cite{BGH05,castelli-devillers-schaeffer,fusy2009transversal,poulalhon-schaeffer-coding}. More recently, they have played a
key role in the study of large random planar maps, culminating with
the existence of a universal scaling limit known as the \emm Brownian
map,~\cite{legall}.

Planar maps equipped with an additional structure (e.g. a spanning tree~\cite{mullin-boisees}, a proper
colouring~\cite{lambda12,tutte-differential}, an Ising or Potts configuration~\cite{baxter-dichromatic,borot3,BK87,mbm-schaeffer-ising,BDG-blocked,daul,eynard-bonnet-potts,Ka86}...) are also much studied,
both in combinatorics and in theoretical physics, where maps are
considered as a model for two-dimensional quantum gravity~\cite{DFGZJ}.
However, for many of these structures, we are still in the early days
of the study, as even their enumeration remains elusive, not to mention
bijections and asymptotic properties.

Recent progresses in this direction include the enumeration of planar
maps weighted by their Tutte polynomial, or equivalently, maps equipped
with a Potts configuration. The associated \gf\ $P(t)$ is known to be \emm
differentially algebraic,. That is, there exists a polynomial equation
relating $P(t)$ and its
derivatives~\cite{bernardi-mbm-alg,BeBM-15}. The Tutte polynomial has
many interesting specializations (in particular, it counts all  structures cited
above, like spanning trees and colourings) and several special cases had been
solved earlier.
One key tool in the solution is that the Tutte
polynomial of a map can be computed inductively, by deleting and
contracting edges. 

Another solved example, which does not {seem to} belong to the Tutte/Potts
realm,
consists of maps (in fact, triangulations) equipped with certain
orientations called \emm Schnyder orientations,.  The results
obtained there have analogies with those obtained for another class of
orientations, called \emm bipolar, (which \emm do, belong to the Tutte
realm). Indeed, for both classes of oriented maps:
\begin{itemize}
\item oriented maps are counted by simple numbers, which are
also known to count other combinatorial objects (various lattice
paths and permutations,  among others);
\item there exist nice bijections explaining these equi-enumeration results~\cite{Boni05,bonichon-mbm-fusy,felsner-baxter,fusy-poulalhon-schaeffer-bip};
\item for a fixed map $M$, the set of Schnyder/bipolar orientations of $M$ has a
  lattice
  structure~\cite{Pro:1993,felsner2004lattice,OdM:1994}. {The
    above bijections, once specialized  to maps equipped
  with their (unique) minimal orientation, coincide with  attractive bijections
  designed earlier for (unoriented)
  maps}~\cite{bernardi-bonichon,bonichon-mbm-fusy};
\item specializing  the bijections further to maps that have only one
  Schnyder/bipolar orientation also yields {interesting combinatorial
  results~\cite{bernardi-bonichon,bonichon-mbm-fusy}.}
\end{itemize}

These observations led us to wonder about another natural class of
orientations, namely  those in which every vertex as equal in- and
out-degree, known as \emm Eulerian, orientations (Figure~\ref{fig:euler}). Clearly, a map needs to be Eulerian to admit an Eulerian
orientation. The condition is in fact sufficient (such maps even
admit an Eulerian circuit~\cite{hierholzer}). One
analogy with the above two classes is that the set of
Eulerian orientations of a given planar map can be equipped with a
lattice structure~\cite{Pro:1993,felsner2004lattice}. Moreover, Eulerian maps  (equivalently,
Eulerian maps equipped with their minimal Eulerian orientation) have
rich combinatorial properties: not only are they counted by simple
numbers (see~\eqref{euler-count}), but they are equinumerous with several other
families of objects, {like certain trees~\cite{mbm-schaeffer-constellations}
and permutations~\cite{bona,fusy-12}.}
And they are  often related to them by beautiful
bijections.

Hence our plan to count Eulerian orientations with $n$ edges. However, this  appears to be a difficult
problem. In fact, we even lack a way
to compute the corresponding numbers  in, say,  polynomial
time. This leads us to resort to approximation methods that are
ubiquitous when studying hard counting problems, like the enumeration
of self-avoiding
walks~\cite{alm-janson,fisher-sykes,guttmann-confinement,poenits-tittmann},
or polyominoes~\cite{barequet-moffie-ribo-rote,jensen,klarner-rivest}:
{denoting by $\PEO$ the set of Eulerian orientations, we
construct subsets and supersets of~$\PEO$, indexed by an integer parameter $k$, which
converge to  $\PEO$ as $k$ increases. And we count the elements of these sets.}

One difference between our study and those dealing with tricky 
objects on regular lattices (like the above mentioned self-avoiding walks and polyominoes) is worth noting.  The subsets
and supersets that are defined to approximate lattice objects  often have a
one-dimensional structure, and  \emm rational
\gfs, that can be obtained using a transfer matrix approach. {A typical
example is provided by self-avoiding walks confined to a strip of
fixed width.} But our
subsets and supersets of orientations belong to the world of maps (or
\emm random lattices,
in the physics terminology), and have \emm algebraic \gfs,. More precisely,
our subsets have a branching, tree-like structure, which yields a
system of algebraic equations for their \gfs, and a universal asymptotic
behaviour in $\lambda^n n^{-3/2}$ (for a growth rate $\la$ depending on
the index $k$). {The \gfs\ of our
supersets are more mysterious. They are bivariate series given by systems of equations}  involving \emm divided differences,
of the form
$$
\frac{F(t;x)-F(t;1)}{x-1},
$$
and we have to resort to a deep theorem in algebra, due to
Popescu~\cite{popescu},  to   prove their
algebraicity for all $k$ (we also solve these systems for
small values of $k$). We conjecture that their asymptotic behaviour is
also universal, this time in $\la^n n^{-5/2}$, as for planar maps
(again, for varying $\la$).
 
\medskip
Here is now an outline of the paper. In Section~\ref{sec:decomp} we
first present a simple recursive decomposition of (rooted) Eulerian orientations,
based on the contraction of the root edge, and then a variant of this
decomposition. Thanks to this variant, we can compute  the number $o_n$ of Eulerian orientations having $n$
edges for $n\le 15$ (Figure~\ref{tab:values}).  By attaching two orientations at
their root vertex, we see that 
the sequence $(o_n)_{n\ge 0}$ is super-multiplicative:
$$
o_{m+n} \ge o_m o_n.
$$
This classically 
implies that the limit $\mu$ of $o_n^{1/n}$ exists
and satisfies
\beq \label{fekete}
\mu=\sup_n o_n^{1/n}.
\eeq
(see Fekete's Lemma in~\cite[p.~103]{vanlint01}).
We call $\mu$ the \emm growth rate, of Eulerian orientations. It is
bounded from below by the growth rate 8 of Eulerian maps, and from
above by the growth rate 16 of Eulerian maps equipped with an
arbitrary orientation. Our data for $n\le 15$ suggest than $\mu$ is  around
$12.5$ (Figure~\ref{tab:values}, right). Using differential
approximants~\cite{tony-perso}, Tony
Guttmann predicts $\mu=12.568\ldots$, and an asymptotic behaviour
{$o_n\sim c \mu^n n^{-\gamma}$ with $\gamma=2.23\ldots$}

\begin{figure}[h]
  \begin{minipage}{9cm}
\hskip -8mm \begin{tabular}{c|c || c | c || c | c } 
$n$ & $o_n$ & $n$ & $o_n$ & $n$ & $o_n$\\
\hline 
0 & 1  & 6 & 37 548 & 12 & 37 003 723 200\\
1 & 2  & 7 & 350 090 & 13 & 393 856 445 664\\
2 & 10 & 8 & 3 380 520 & 14 & 4 240 313 009 272\\
3 & 66 & 9 & 33 558 024 & 15 & 46 109 094 112 170\\
4 & 504 & 10 & 340 670 720 & & \\
5 & 4 216 & 11 & 3 522 993 656 & &\\
\end{tabular}   
  \end{minipage} \begin{minipage}{4cm}
     \includegraphics[scale=0.5]{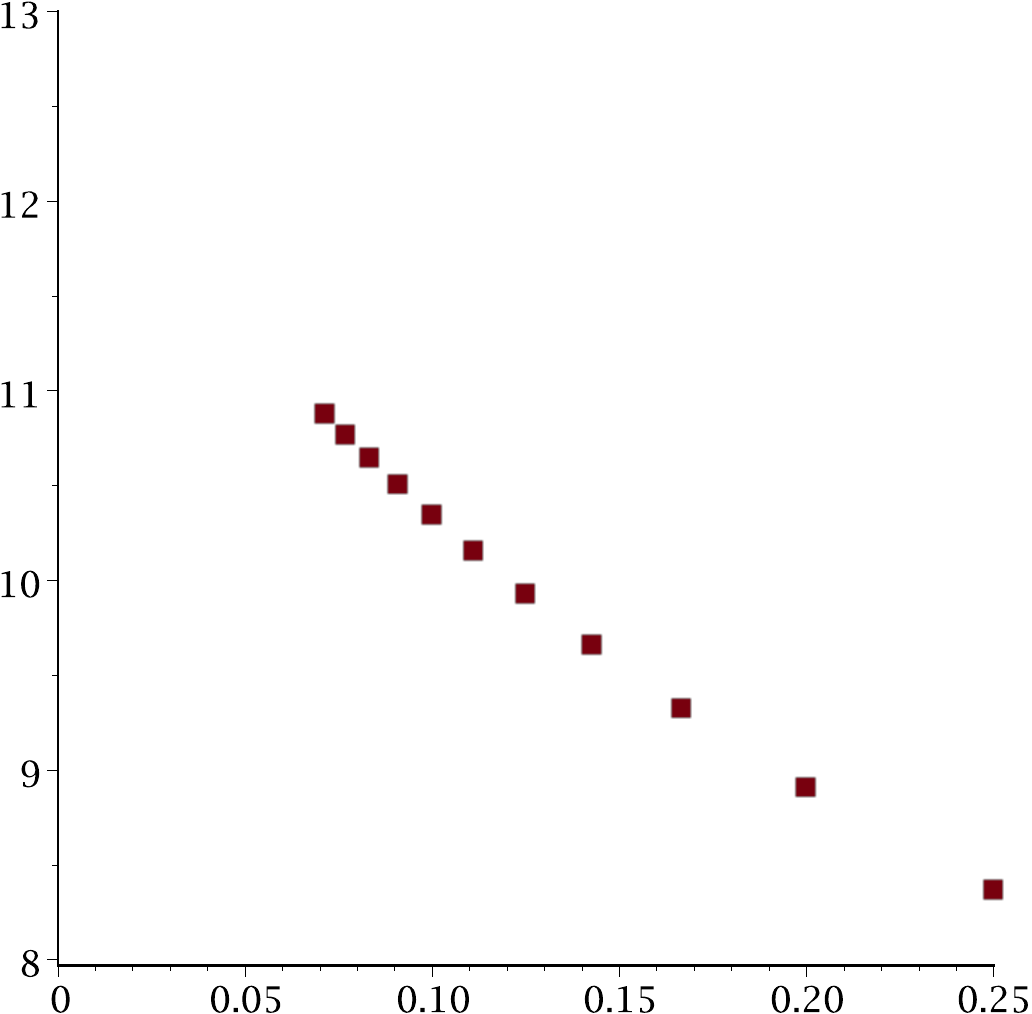}
  \end{minipage}

 \hskip 5mm
\caption{\emph{Left:} First values of $o_n$, for $n$ from 0 to
  15 (entry A277493 of the OEIS~\cite{oeis}). \emph{Right:} {A plot of $o_{n+1}/o_n$ vs. $1/n$, for $n=4, \ldots, 14$,
    suggests that the growth rate of Eulerian orientations, located
    at the intercept of the curve and the $y$-axis, is
    around $12.5$.}}
\label{tab:values}
\end{figure}

Sections~\ref{sec:classic-lower} and~\ref{sec:prime-lower} deal with two families of
subsets of Eulerian orientations. The first family uses our first
recursive decomposition of Eulerian orientations, and should be
considered as a warm up.  The second family
uses the variant of the standard decomposition of orientations. Its study is a bit more involved, but it gives
better bounds on the growth constant.  Both families are proved to have algebraic
\gfs\ and a tree-like asymptotic behaviour in $\la^n n^{-3/2}$. The
next two sections deal with two families of
supersets of Eulerian orientations. Both are proved to have algebraic
\gfs, and we conjecture a map-like asymptotic behaviour in  $\la^n
n^{-5/2}$. We solve our systems of equations explicitly for small
values of $k$, and thus obtain  Table~\ref{tab:general}.
All calculations are supported by {\sc Maple} sessions
  available on our web pages.
We gather in Section~\ref{sec:final} a few final comments and questions.

\begin{table}[h]
\caption{Growth rates and
  cardinalities of subsets ($\mathcal L^{(k)}$ and $\mathbb L^{(k)}$) and supersets ($\mathcal U^{(k)}$ and
  $\mathbb U^{(k)}$) of
  Eulerian orientations. 
The table also records the degrees of the associated \gfs, which are
  systematically algebraic. The symbol $\simeq$ refers to a numerical
  estimate. The other growth rates are algebraic numbers known
  exactly via their minimal polynomial.}

\begin{tabular}{|c|c|c|c|c|c|c|c|c|c|}
\hline
 & degree & growth & 1 & 2 & 3 & 4 & 5 & 6&7\\
\hline
Eulerian maps &2 & $8$ & 1& 3& 12& 56 & 288 & 1 584 &9 152\\
\hline
$\mathcal{L}^{(1)}$&2 &$9.68\ldots$ & \bf 2 &\bf 10&\bf 66 &466 &3 458
&26 650& 211 458\\
$\mathcal{L}^{(2)}$&4  & $10.16\ldots$ & \bf 2 &\bf 10 &\bf 66 &\bf 504
&4 008 &32 834&275 608 \\
\hline
$\mathbb{L}^{(1)}$ &3 & $ 10.60\ldots$ & \bf 2 & \bf 10 & \bf 66 & 490 &
3 898 & 32 482& 279 882\\
$\mathbb{L}^{(2)}$ & 6 & $10.97\ldots$ & \bf 2 & \bf 10 & \bf 66 & \bf
504 & 4 148 & 35 794 &319 384\\
$\mathbb{L}^{(3)}$ &20 & $11.22\ldots$ & \bf 2 & \bf 10 & \bf 66 & \bf
504 & \bf 4 216 & 37 172 & 339 406 \\
$\mathbb{L}^{(4)}$&258  & $\simeq 11.41 $&\bf 2& \bf  10& \bf  66& \bf  504&
\bf  4 216& \bf  37 548&347 850\\
$\mathbb{L}^{(5)}$& ?  &$\simeq 11.56$&\bf 2& \bf  10& \bf  66& \bf  504& \bf
4 216& \bf  37 548& \bf 350 090\\
\hline
\textbf{Eulerian orientations} && $\boldsymbol{\simeq 12.5}$ & \textbf{2} & \textbf{10} & \textbf{66} &
\textbf{504} & \textbf{4 216} & \textbf{37 548} &\bf 350 090\\
\hline
$\mathbb{U}^{(5)}$ & ? & ${\simeq 13.005}$&\bf 2& \bf  10& \bf  66& \bf  504& \bf
4 216& \bf  37 548& \bf 350 090\\
$\mathbb{U}^{(4)}$ & ? & ${\simeq 13.017}$  &\bf 2& \bf 10&\bf 66 &\bf 504 &\bf 4
216& \bf 37 548 & 350 538 \\
$\mathbb{U}^{(3)}$ & ? &{$\simeq 13.031$ } &\bf 2& \bf 10&\bf 66 &\bf 504 &\bf 4
216& 37 620 & 352 242 \\
$\mathbb{U}^{(2)}$ & 28& $13.047\ldots$ &\bf 2& \bf 10&\bf 66 &\bf 504 &4
228&37 878  & 356 252\\
$\mathbb{U}^{(1)}$&3 & $13.065\ldots$ & \bf 2 & \bf 10& \bf 66 & 506 & 4
266 & 38 418 &363 194\\
\hline
$\mathcal{U}^{(2)}$ & 27&$13.057\ldots$ & \bf 2 & \bf 10 & \bf 66 & \bf
504 & 4 232  & 37 970 &357 744\\
$\mathcal{U}^{(1)}$ &3& $13.065\ldots$ & \bf 2 & \bf 10& \bf 66 & 506 &
4 266 & 38 418 &363 194\\
\hline
Oriented Eulerian maps &2& $16$ & \bf 2 & 12 & 96 & 896 & 9 216 & 101
376 &1 171 456\\
\hline
\end{tabular}
\smallskip
\label{tab:general}
\end{table}

Let us mention that counting Eulerian orientations  of 4-valent
(rather than Eulerian) maps might be simpler: in this case, the number of Eulerian orientations is
  a specialization of the Tutte polynomial~\cite{welsh-tutte}, and in fact some
 results   exist in the physics literature~\cite{kostov,zinn-justin-6V-random}.  In
  the final section, we discuss  further this problem, which seems to
  deserve more attention.

\section{Recursive decompositions of Eulerian orientations}
\label{sec:decomp}
In this section, after a few basic definitions, we recall the standard recursive decomposition of
Eulerian \emm maps, based on the contraction of the root edge, which can be
traced back to the early papers of Tutte (e.g.~\cite{tutte-general}). We then adapt it to
decompose Eulerian \emm orientations,. We also introduce a variant of the
standard decomposition of Eulerian maps, based on a notion of \emm
prime, maps, and adapt it again to Eulerian orientations. This variant
is slightly more involved, but
turns out to be more effective: it allows us to compute the numbers
$o_n$ for larger values of $n$, and it also leads to
better lower and upper bounds on these numbers (Sections~\ref{sec:prime-lower}
and~\ref{sec:prime-upper}).

\subsection{Definitions}

A \emph{planar map} is a proper
 embedding of a connected planar graph in the
oriented sphere, considered up to orientation preserving
homeomorphism. Loops and multiple edges are allowed
(Figure~\ref{fig:euler}). The \emph{faces} of a map are the
connected components of  its complement.  The number of
 edges  of a planar map $M$ is denoted by 
$\ee(M)$.
 The \emph{degree} of a vertex 
 is the number
of edges incident to it, counted with multiplicity (e.g., a loop
counts twice).
A \emph{corner} is
a sector delimited by two consecutive edges around a vertex;
hence  a vertex 
 of degree $k$ defines $k$ corners. 
Alternatively, a corner can be described as an incidence between a
vertex and a face. 

For counting purposes it is convenient to consider \emm rooted, maps. 
A map is rooted by choosing a corner, called  the \emm root corner,.
The vertex and face that are incident at this corner are respectively
the \emm root vertex, and the \emm root face,.
 The \emm root edge, is the edge that follows the
root corner in counterclockwise 
order around the root vertex.
 In figures, we  indicate the rooting by 
 an arrow pointing to the root corner, and take the root  face
as the infinite face (Figure~\ref{fig:euler}). 

From now on, every {map}
is \emph{planar} and \emph{rooted}, and these precisions will often be
omitted. By convention,
we include among rooted planar maps the \emph{atomic map}
 having one vertex and no edge.

A map $M$ is  \emph{Eulerian} if every vertex has even
degree. In this case,  we denote by $\dv(M)$ the half degree of the
root vertex.  An \emph{Eulerian orientation} is a map with  oriented
edges,  in
which the in- and out-degrees of every vertex are equal. Note that
the underlying  map must be Eulerian. We denote by
$\mathcal M$ the set of  Eulerian maps, and by $\PEO$ the set of Eulerian orientations.

\subsection{Eulerian maps: standard decomposition}
\label{sec:dec-map-standard}

Consider  an Eulerian map $M$, not reduced to the atomic
map, and its root edge $e$.  If $e$ is a loop, then $M$ is obtained from two smaller 
maps $M_1$ and $M_2$ by joining $M_1$ and $M_2$ at their root vertices
and adding  a loop surrounding $M_1$ (Figure~\ref{fig:dec_eul},
left). The maps $M_1$ and $M_2$ are themselves Eulerian (because the
sum of vertex degrees in a map is even, so that one cannot have a
single odd vertex {in $M_1$ or $M_2$}). We call this
operation the \emm merge, of $M_1$ and $M_2$. 

If the root edge $e$ is not a loop, then we contract it, which gives
 a smaller Eulerian map $M'$. Note however that several maps
give $M'$ after contracting  their root edge. All such maps can be
obtained from $M'$ as follows (see Figure~\ref{fig:dec_eul}, right): we split the root vertex $v$ of $M'$
into two vertices $v$ and  $v'$ joined by an edge (which will be the
root edge), and distribute the edges adjacent to
$v$ between $v$ and $v'$ in such a way the degrees of $v$ and $v'$
remain even. This operation is called a \emm split, of $M'$, and more
precisely an \emm $i$-split, if $v$ has degree $2i$ in the larger
map. Note that if 
$v$ has degree $2d$ in $M'$, then $i$ can be chosen arbitrarily
between 1 and $d$.

Let $M(t;x)$ be the generating function of Eulerian maps, counted by edges (variable $t$) and by the {half degree} of the root vertex  (variable $x$): 
$$
M(t;x) = \sum_{M \in \mathcal{M}} t^{\ee(M)} x^{\dv(M) } = \sum_{d \geq 0} x^d M_{d}(t),$$
where 
$M_d(t)$ denotes the edge \gf\ of  Eulerian maps with root vertex degree~$2d$.
The above construction translates into the following functional
equation, which we explain below:
\begin{align} 
  M(t;x)&= 1 + tx M(t;x)^2 + t\sum_{d\geq 0} M_{d}(t) (x+x^2 + \cdots
  + x^d) \nonumber \\
&= 1 + tx M(t;x)^2 + t\sum_{d\geq 0} M_{d}(t) \frac{x^{d+1}-x}{x-1}\nonumber
\\&= 1 + tx M(t;x)^2 + \dfrac{tx}{x-1}(M(t;x) - M(t;1)).\label{eq-func-M}
\end{align}
On the first line, the term 1 accounts for the atomic map, the next
term for maps obtained by merging two smaller maps, and the third term
for maps obtained from a split.

\begin{figure}[h]
\centering
\includegraphics[scale=0.7]{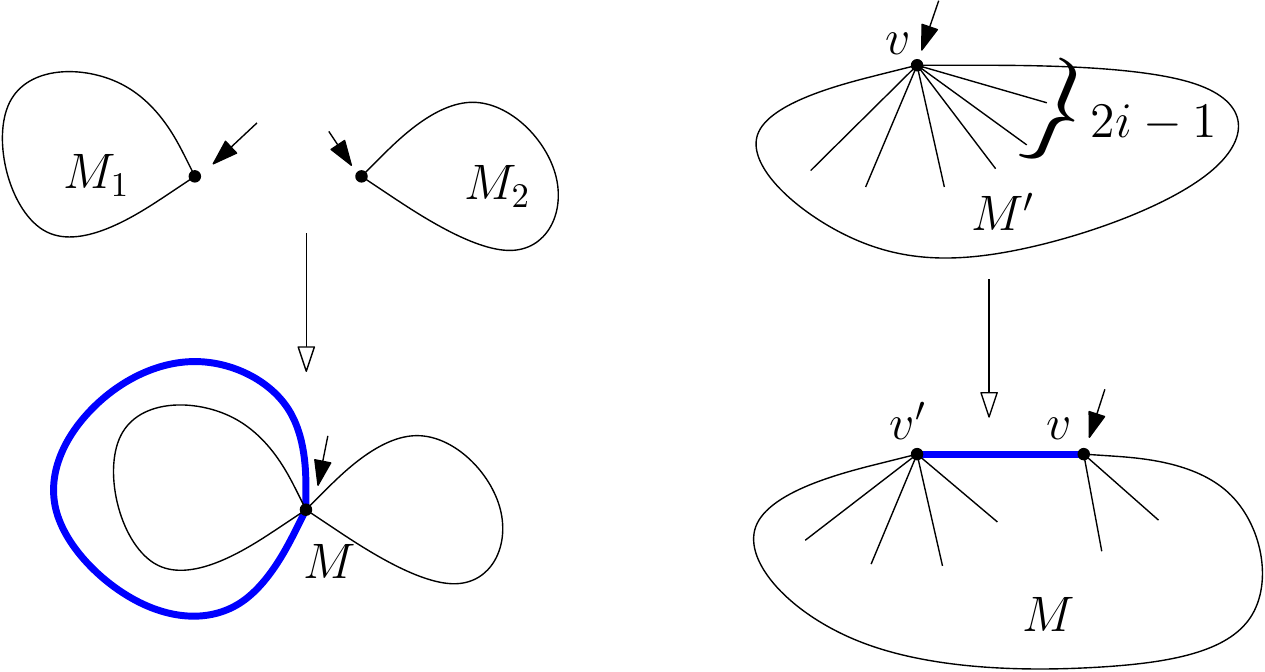}
\caption{Construction of  an Eulerian map with $n$ edges: merge
    an ordered pair of Eulerian maps $M_1$, $M_2$ with $n_1$ and $n_2$
    edges ($n_1 + n_2 =n-1$) and add a loop, or make a split on
    an Eulerian map with $n-1$ edges. The new edge (here thicker) is the root edge of $M$.}
\label{fig:dec_eul}
\end{figure}

\subsection{Eulerian orientations: standard decomposition}
\label{sec:dec-orientation-standard}
Our recursive decomposition for Eulerian orientations is essentially
the same as  for Eulerian maps: if the root edge is a loop, we
delete it and obtain two  orientations, which are both Eulerian
(in any oriented map, the sum over all vertices  of  in-degrees equals the sum of 
out-degrees, hence one cannot have a single vertex with distinct in-
and out-degrees);   otherwise we contract
the root edge, which gives a smaller Eulerian orientation.

However, care must be taken when going in the opposite direction, that
is, when constructing  large orientations from smaller ones.
The first type of orientations, obtained by a merge, do not raise any
difficulty; one can orient the new root edge (the loop) in two
different ways (Figure~\ref{fig:peo}, left). 
But consider now an Eulerian orientation $O'$, with root vertex $v$ of
degree at least $2i$, and   perform an $i$-split on $O'$: is there a
way to orient the new edge so as to obtain an orientation $O$ that is
still Eulerian? The answer is yes  if and only if  the numbers of in- and
out-edges in the last  $2i-1$ edges incident to $v$ in $O'$ differ by
 $\pm 1$ (edges are visited in counterclockwise order, starting from
the root corner). The orientation of the root
edge of $O$ is then forced (Figure~\ref{fig:peo}, right). In this case, we
say that the $i$-split, performed 
on $O'$, is  \emm legal,.   Note that the  $1$-split and the $d$-split
are always legal, where $2d$ is the degree of the root vertex of $O'$.

The fact that not all splits are legal makes it difficult to write a
single functional equation for the \gf\ of Eulerian
orientations. However, we can write an \emm infinite system, of
equations relating the \gfs\ of orientations with prescribed
orientations at the root. 

Let us be more precise. 
 Given an Eulerian orientation $O$ with root vertex $v$ of
 degree $2d$, the \emph{root word} $w(O)$ of $O$  is a word of length
 $2d$ on the alphabet $\{0,1\}$  describing the orientation of the
 edges around $v$ (in counterclockwise order, starting from the root
 corner):  the $k$-th letter of $w(O)$ is 0 (resp. 1) if the $k$-th
 edge around $v$ is  in-going (resp. out-going). Note that this word
 is always \emph{balanced}, meaning that  it contains as many 0's {as}
 1's. We call a word $\gw$ \emph{quasi-balanced} if the number of 0's
 and 1's in $\gw$ differ by $\pm 1$. The length (number of letters) of
 $\gw$ is denoted by $|\gw|$, while the number of occurrences of
 the letter $a$ in it is denoted by $|\gw|_a$. We define the
 \emph{balance} of  $\gw$ to be $b(\gw) := ||\gw|_1 - |\gw|_0|$. The empty word is denoted by $\vareps$.

Now we can decide from the root word of $O'$ if the $i$-split of $O'$ is
legal: this holds if and only if the last $2i-1$ letters of $w(O')$
form a quasi-balanced word.

\begin{figure}[h]
\centering
\includegraphics[scale=1]{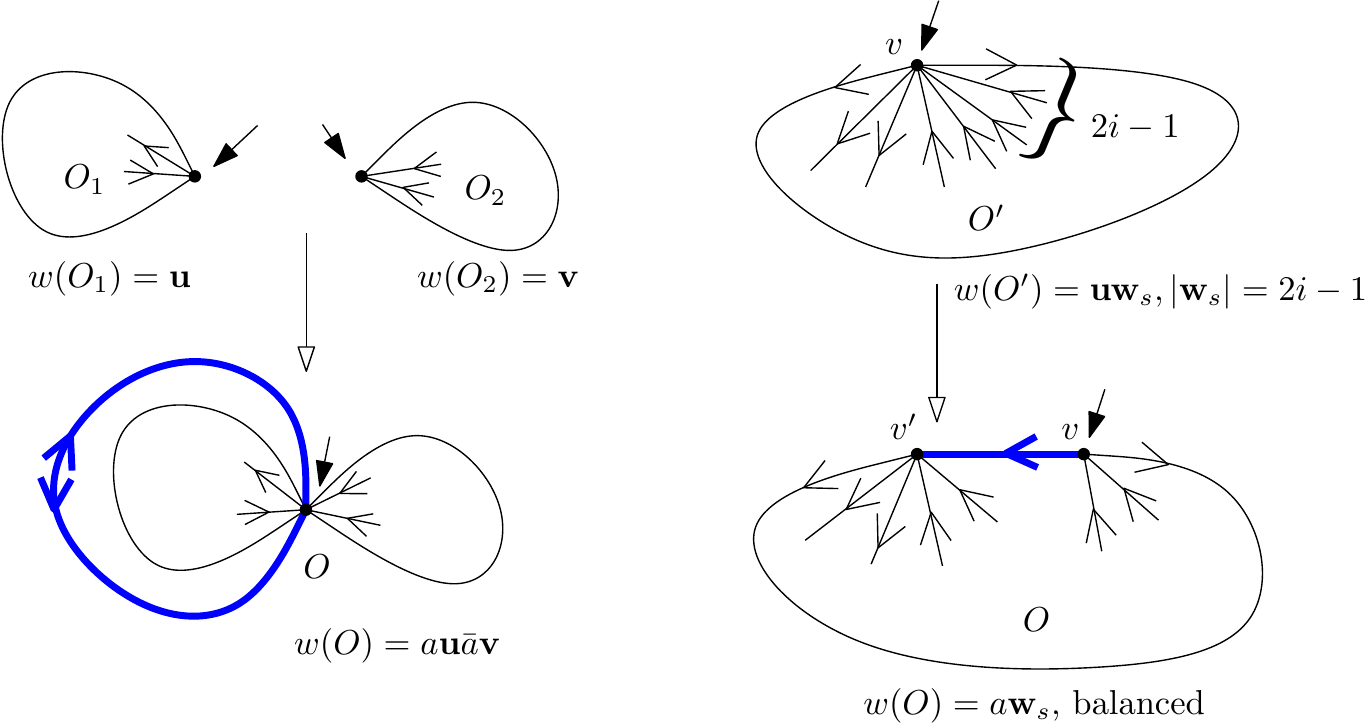}
\caption{Construction {of} an Eulerian orientation: 
    merge two Eulerian orientations (the loop  can be
    oriented in either way), or split (legally) an Eulerian
    orientation. Observe  how the root word changes.}
\label{fig:peo}
\end{figure}

For $\gw$ a word on $\{0,1\}$, let $O_\gw(t)\equiv O_\gw$ be the generating
function of Eulerian orientations having $\gw$ as root word, counted
by their edge number. Clearly,
$O_\gw=0$ if $\gw$ is not balanced and $O_\vareps=1$ (accounting for
the atomic map). Now if $\gw$ is non-empty and balanced,
\beq\label{eq-O}
O_\gw=	t \sum_{a\gu\ba \gv=\gw} O_{\gu} O_{\gv} + t \sum_{\gu}
O_{\gu\gw_s} .
\eeq
This identity is illustrated in Figure~\ref{fig:peo}.
Here, $a$ stands for any of the letters 0, 1, and the first sum runs
over all factorisations of $\gw$ of the form $a\gu\ba \gv$, with
$\bar{a}:=1-a$. This sum
counts orientations obtained by a merge.
The second sum runs over all possible words $\gu$, and $\gw_s$ denotes
the suffix of $\gw$ of length $|\gw|-1$. This sum counts
orientations obtained by a (legal) split of an orientation having root word
$\gu \gw_s$. 
Now  the \gf\ $O$ of Eulerian orientations is
$$
\sum_\gw
O_\gw ,
$$
where the sum runs over all (balanced) words $\gw$.

 We do not know how to solve this system.  But a map with $n$
edges has a root word of length at most $2n$, and hence we
can use our system to compute the numbers $o_n$ for $n$ small. We
obtain in this way the first 11 values of
Figure~\ref{tab:values}.

In Sections~\ref{sec:classic-lower} to~\ref{sec:prime-upper}, we define subsets and supersets of
$\PEO$ that we can generate by just looking at the last $2k-1$ letters
of the root word (for $k$ fixed). This allows us to write finitely
many equations for the \gfs\ of these subsets and supersets. Solving
them gives lower and upper bounds on the growth rate of Eulerian
orientations. However, we obtain more precise bounds by using a
variant of the standard  decomposition of maps and orientations. We
now present this variant.

\subsection{Prime decomposition of maps and orientations}
\label{sec:prime-dec}
A (non-atomic) map is said to be \emph{prime} if the root vertex appears only once when
walking around  the root face. A  planar map $M$ can  be seen as  a
sequence of prime maps $M_1,  \ldots, M_\ell$
(Figure~\ref{fig:prime_decomp}).  We say that the $M_i$ are the \emm prime
submaps, of $M$, and denote  $M = M_1 \cdots M_\ell$. Note that if $M$
is Eulerian, then each $M_i$ is Eulerian too.

\begin{figure}[h]
\centering
\includegraphics[scale=0.7]{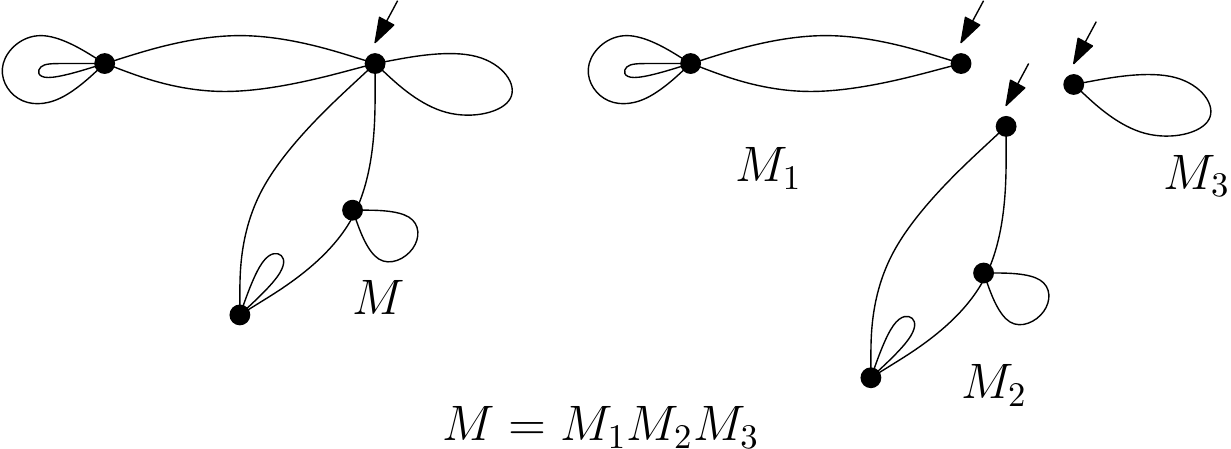}
\caption{{Decomposition of an Eulerian map $M$ into prime Eulerian maps $M_1$, $M_2$, $M_3$.}}
\label{fig:prime_decomp}
\end{figure}

Now take a \emm prime, Eulerian map $M$, and apply the standard decomposition of
Section~\ref{sec:dec-map-standard}, illustrated in Figure~\ref{fig:dec_eul}: either $M$ is an (arbitrary) Eulerian map $M_1$
surrounded by a loop, or $M$ is obtained by an $i$-split in another
Eulerian map $M'$, \emm provided the last prime submap of $M'$ (in
counterclockwise order) has
root degree at least $2i$, (otherwise, the resulting map
  would not be prime). Alternatively, if $M'= M_1' \cdots M_\ell'$,
we can obtain $M$ by performing an $i$-split in the prime map $M_\ell'$, and
attaching the map $M'':=M_1' \cdots M_{\ell-1}'$ at the new vertex $v'$
created by this split (Figure~\ref{fig:rec-prime}).

\begin{figure}[h]
\centering
\includegraphics[scale=1]{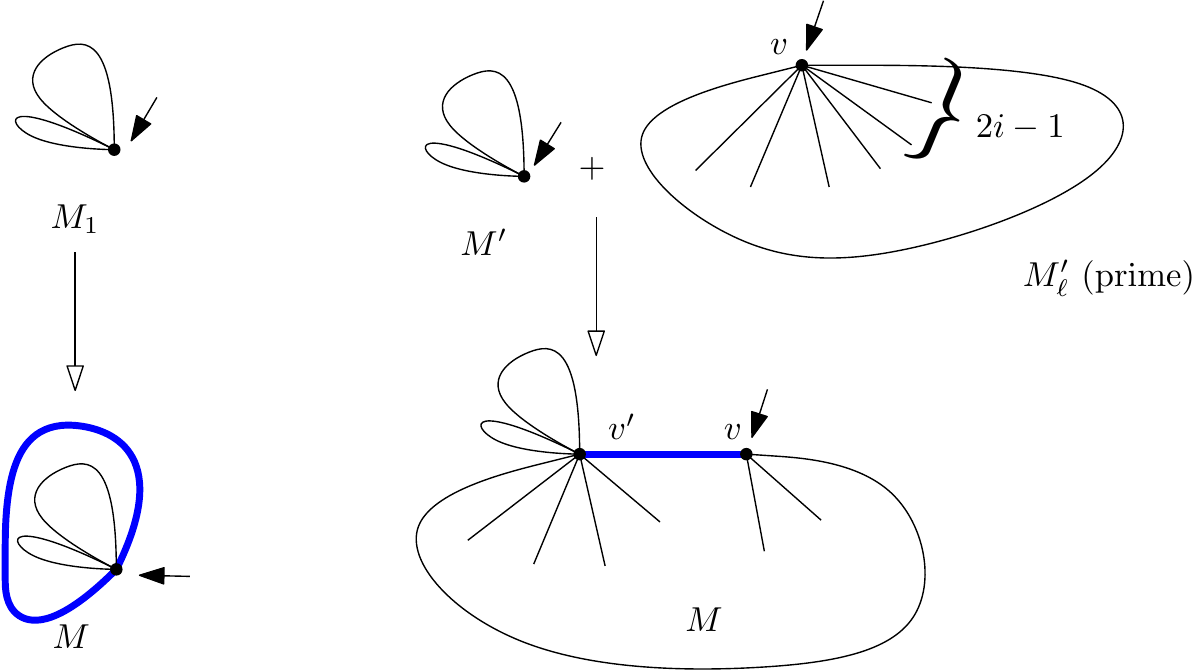}
\caption{Construction of a prime Eulerian map: add a loop around any
  Eulerian map, or split a  prime Eulerian map, and attach an arbitrary
  Eulerian map at the end of its root edge.}
\label{fig:rec-prime}
\end{figure}

This alternative decomposition of Eulerian maps gives a system of two
equations defining the \gf\  $M(t;x)$ of Eulerian maps (still counted
by edges and root vertex degree) and its counterpart $M'(t;x)$ for prime maps:
\begin{align*} 
M(t;x) &= 1+ M(t;x) M'(t;x), \\
M'(t;x) &= tx M(t;x) + t xM(t;1) \dfrac{M'(t;x) - M'(t;1)}{x-1}.
\end{align*}
In the first equation, the term $M'(t;x)$ accounts for the last prime
submap attached at the root vertex (denoted $M_\ell$ above). In the second equation, the divided difference $(M'(t;x)-M'(t;1))/(x-1)$ has the same explanation as
in~\eqref{eq-func-M}. This equation is easily recovered by eliminating $M'(t;x)$
from the above system.

\medskip
This decomposition can also be applied to Eulerian orientations: an
Eulerian orientation is a sequence  of prime Eulerian orientations,
and a prime orientation is either obtained by adding an oriented loop
around another orientation, or by performing a legal split in a prime
orientation, and attaching another orientation at the vertex $v'$
created by the split.

Thus, denoting again by  $O_\gw$ the \gf\ of orientations with root
word $\gw$, and by $O'_\gw$ its counterpart for prime orientations, we
now have $O_\gw=O'_\gw=0$ if $\gw$ is not balanced, $O_\vareps=1$,
$O'_\vareps=0$ and finally for $\gw$ balanced and non-empty,
\begin{align*}
  O_\gw&= \sum _{\gu \gv=\gw}O_\gu O'_\gv,
\\
O'_\gw &= t O_{\gw_c}+ t O \sum_{\gu} O'_{\gu\gw_s}.
\end{align*}
In the second equation,  $\gw_c$ denotes the central factor or
$\gw$ of length $|\gw|-2$, and $O=\sum_\gw O_\gw$ is the \gf\ of
all Eulerian orientations. Recall that $\gw_s$ is the suffix of $\gw$
of length $|\gw|-1$.

Using these equations, we have been able to push further the
enumeration of Eulerian orientations of small size, thus obtaining the
values of Figure~\ref{tab:values}.

\section{Subsets of Eulerian orientations, via the standard decomposition}
\label{sec:classic-lower}
In this section and the three following ones, we define certain  subsets and supersets of Eulerian
orientations, indexed by an integer $k$, which converge (monotonously)
to the set
$\PEO$ of all Eulerian orientations as $k$ tends to infinity.
Those sets are respectively denoted by $\peom$ (and $\pinf$), $\peop$ (and $\psup$), as they give
 \emm lower, and \emm  upper, bounds on the numbers $o_n$ and their
 growth rate. For each set, we give a system of functional
 equations defining its \gf: for
the subsets $\peom$ and $\pinf$, these systems are algebraic, so that the
associated \gfs\ are algebraic series. For the supersets $\peop$ and
$\psup$, the systems define 
bivariate series  and involve divided differences as
in~\eqref{eq-func-M}. However, we prove that the resulting series are also
algebraic.

\medskip
Recall from Figure~\ref{fig:peo} that 
planar Eulerian orientations can be  obtained recursively from the atomic map by either:
\begin{itemize} 
\item the merge of two orientations $O_1, O_2 \in \PEO$ (with the root
  loop oriented in either way),
\item or a legal split on an orientation $O' \in \PEO$.
\end{itemize}

\begin{definition} \label{def:m}
Let $k\ge 1$. Let $\peom$ be the set of planar orientations obtained
recursively from the atomic map by either:
\begin{itemize} 
\item the merge of two orientations $O_1, O_2 \in \peom$ (with the root
  loop oriented in either way),
\item or a legal $i$-split on an orientation $O' \in \peom$ such that  $i\leq k$ or $i=\dv(O') $.
\end{itemize}
In other words, the only allowed splits are the \emm small, splits
($i\le k$) and the \emm maximal, split ($i=\dv(O') $).
\end{definition}

Obviously, all orientations of $\peom$ are Eulerian. Moreover, the sets $\peom$
form an increasing sequence since more and more (legal) splits are
allowed as $k$ grows. Finally, all Eulerian orientations of
  size $n$ belong to $\mathcal L^{(n)}$ (and even to $\mathcal L^{(n-2)}$). Hence  the
limit of the sets $\peom$ is the set $\PEO$ of all Eulerian
orientations.

{Figure~\ref{fig:inf-1} shows a (random) orientation of $\mathcal
L^{(1)}$.}

\begin{figure}[h]
\centering
\includegraphics[angle=90,scale=0.85]{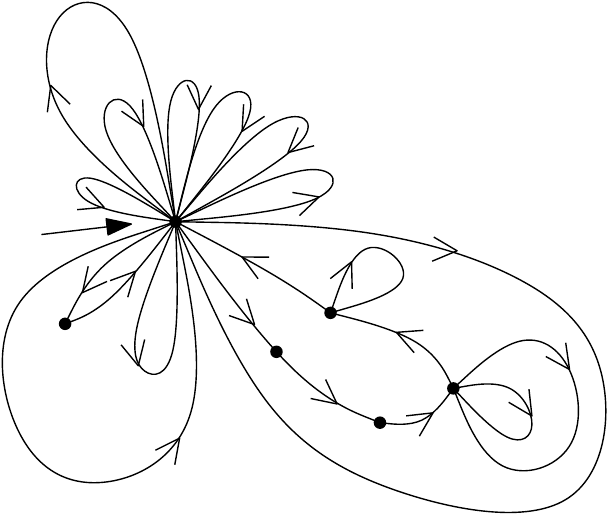}
\caption{An Eulerian orientation in $\mathcal L^{(1)}$, taken
  uniformly at random among those with $20$ edges.}
\label{fig:inf-1}
\end{figure}

\subsection{An algebraic system for $\boldsymbol\peom$}
\label{sec:alg-Lc}
In this section, $k$ is a fixed integer.
\begin{definition}\label{def:valid}
   A word $\gw$ on $\{0,1\}$ is  \emph{valid} (for $k$) if
there exists  a balanced word of length $2k$ having $\gw$ as a
factor. Equivalently,
the balance of $\gw$ satisfies $b(\gw) \leq 2k-|\gw|$. This holds
automatically if $|\gw| \le k$.
\end{definition}
Given a word $\gw$, it will be convenient to have notation for several
words that differ from $\gw$ by one or two letters. We have already
defined $\gw_c$, the central factor of $\gw$ of length
$|\gw|-2$, and $\gw_s$, the suffix of $\gw$ of length $|\gw|-1$. We 
similarly define $\gw_p$ as the  prefix of $\gw$ of  length $|\gw|-1$. Finally, if $\gw$ is quasi-balanced, then
$\overleftarrow{\gw}$ stands for the unique balanced word of the form $a\gw$,
for $a\in \{0,1\}$.

For any word $\gw$, we denote  by $\Lc^{(k)}_\gw(t)$ 
the generating function of orientations of $\peom$  whose root word ends with $\gw$, counted by
edges. In particular, the \gf\ counting all orientations of $\peom$ is
$\Lc^{(k)}_\vareps(t)$. We denote by $\Kc^{(k)}_\gw(t)$
 the generating
function of orientations of $\peom$ having root word exactly~$\gw$. In order to
lighten notation, we often
omit the dependence of our series in  $t$ and the superscript $(k)$.

We now give equations defining the series $\Lc_\gw$ (for  $|\gw|\le 2k-1$) and
$\Kc_\gw$ (for $|\gw|\le 2k$). First, we note that $\Kc_\gw=0$ if $\gw $
is not balanced, and that $\Kc_\vareps=1$. Now for $\gw$  balanced of
length between 2 and  $2k$, we have:
\begin{equation}\label{F-l-c}
  \Kc_\gw=	t \sum\limits_{\gw =a \gu \bar{a} \gv }
 \Kc_{\gu} \Kc_{\gv} + t \Lc_{\gw_s},
  \end{equation}
where, as before, $a$ is any of the letters $0,1$.
This equation is analogous
to~\eqref{eq-O}: the first term counts orientations obtained
from a merge, the second  orientations obtained from a split.
Now for $\Lc_\gw$, with $\gw$ of length at most $2k-1$, we have:
\begin{multline}\label{L-l-c2}
  \Lc_\gw= \mathbbm{1}_{\gw=\vareps}
+2 t \Lc_\varepsilon \Lc_\gw + 
	t\sum\limits_{\gw = \gu a \gv}
\Lc_{\gu} \Kc_{\gv} +
	t\sum\limits_{\gw = a \gu \bar{a} \gv }
 \Kc_{\gu} \Kc_{\gv} \\+ 
	t \left(\Lc_\gw - \mathbbm{1}_{\gw=\vareps}\right)
+ 	t\sum\limits_ {\substack{\gu = \gv \gw \\
 	2\le|\gu| \le 2k \\ \gu \text{ balanced}}} (\Lc_{\gu_s} -
 	\Kc_{{\gu}})  
.
   \end{multline}
This equation deserves some explanations.  The first line counts
the atomic map (if $\gw=\vareps$), and the orientations  obtained by a merge. The
second (resp. third, fourth) term
of this line counts orientations such that no (resp. one, both) half-edge(s) of
the root loop is/are involved in the suffix $\gw$ of the root
word. Equivalently, denoting by $O_1$ and $O_2$ the merged
orientations, 
those three terms respectively correspond to $|\gw|\le 2\dv(O_2)$, $
2\dv(O_2)<|\gw| < 2+ 2\dv(O_1) + 2\dv(O_2)$ and $|\gw| =2+ 2\dv(O_1) +
2\dv(O_2)$. 

The second line counts orientations $O$ obtained by a legal $i$-split
in a smaller orientation~$O'$. The first term 
 accounts for maximal
splits ($i=\dv(O')$), which, we recall, do not change the root
word (note also that no split is possible on the atomic map). The second term  counts
orientations $O$  obtained from a
non-maximal split. The
word $\gu$ stands for the  root word of
$O$.  The subtraction of $
\Kc_{{\gu}}$ comes from the condition that the split is
not maximal.

\begin{prop}\label{prop:l-c}
  Consider the collection of equations consisting of:
  \begin{itemize}
  \item Equation~\eqref{F-l-c}, written for all balanced words $\gw$
    of length between $2$ and $2k$,
\item Equation~\eqref{L-l-c2}, written for all  valid words $\gw$
of  length at most $ 2k-1$.
  \end{itemize}
In this collection, replace all trivial $\Kc$-series by their value:
$\Kc_\gw=0$ when $\gw$ is not balanced, $\Kc_\vareps=1$. Let $S_0$
denote the resulting system. The number of
series it involves is
\beq\label{ek-def}
\e(k)={2k+2 \choose k+1}-1+\sum_{i=1} ^{k-1} {2i\choose i}.
\eeq
The system $S_0$ defines uniquely these $\e(k)$ series. Its size can be (roughly) divided by two upon noticing that replacing
all $0$'s by $1$'s, and vice-versa, in a word $\gw$, does not change
the series $\Lc_\gw$ nor $\Kc_\gw$. 
\end{prop}
\begin{proof}
  To see that $S_0$ defines all the series that it involves, it
  suffices to note the factor $t$ in the right-hand sides of~\eqref{F-l-c} and~\eqref{L-l-c2}, and to check that each series occurring in the right-hand side
  of some equation also occurs as the left-hand side of another. This
  is readily done, as any factor of a valid word is still valid.

Let us now count the equations of the system. The number of non-empty
balanced words of length at most $2k$ is 
$$
\sum_{i=1}^k {2i\choose i}.
$$
 Then,
all words of length at most $k$ are valid, while the number of
valid words  of
length $k+i$, for $1\le i \le k-1$, is
$$
\sum_{j=i}^k {k+i \choose j}.
$$
(One can interpret $j$ as the number of occurrences of $0$ in the
word.) Hence the number of equations in the system is
$$
\e(k)=\sum_{i=1} ^k {2i\choose i} + \sum_{i=0}^k 2^i + \sum_{i=1}^{k-1} \sum_{j=i}^k
  {k+i \choose j}. 
$$
The second sum evaluates to $2^{k+1}-1$. The third one is
\begin{align*}
  \sum_{j=1}^k \sum_{i=1}^{\min(j,k-1)}  {k+i \choose j}&=
\sum_{j=1}^{k-1} \sum_{i=1}^{j}  {k+i \choose j}+ \sum_{i=1}^{k-1}
    {k+i \choose k}\\
&=
\left(1+ {2k+1 \choose k} - 2^{k+1} \right) +\left( \frac k{k+1}
    {2k\choose k} -1\right)
\\&= \frac {3k+1}{k+1}{2k \choose k} -2^{k+1}.
\end{align*}
The sums are evaluated using classical summation identities, or
Gosper's algorithm~\cite{AB}. The expression of $\e(k)$ given in the
proposition then  follows after elementary manipulations.
\end{proof}

\medskip
\begin{remark}\label{forward}
If $\gw$ is such that $0\gw$ and $1\gw$ are both valid of length less
than $2k$,  we can 
define $\Lc_\gw$ by a simpler  "forward" equation, without increasing the
size of the system:
\begin{equation}\label{L-l-c1}
\Lc_\gw= \Kc_\gw + \Lc_{0\gw}+ \Lc_{1\gw}.
\end{equation}
This is obviously smaller than~\eqref{L-l-c2}, and possibly better suited to
feed a computer algebra system.
However, mixing equations of type~\eqref{L-l-c2} and~\eqref{L-l-c1}  makes
some proofs of Section~\ref{sec:asympt-c} heavier. 
\end{remark}
\subsection{Examples}\label{sec:ex-c}
\subsubsection{When $\boldsymbol{k=1}$,}
the system $S_0$ 
contains $\e(1)=5$ equations and reads
\beq\label{syst-L-1}
\begin{cases}
\Kc_{01} \!\! \!\!&  =\  t\Kc_\varepsilon \Kc_\varepsilon + t\Lc_1, \\
\Kc_{10} \!\!\!\! &  =\  t\Kc_\varepsilon \Kc_\varepsilon + t\Lc_0,
\\
\Lc_\varepsilon \!\! &  =\  1+ 2t \Lc_\vareps \Lc_\vareps+ t(\Lc_\vareps-1) + t (\Lc_0-K_{10}+\Lc_1-K_{01}),
\\
\Lc_0 \!\! &  =\  2t \Lc_\varepsilon \Lc_0 + t\Lc_\varepsilon \Kc_\varepsilon + t\Lc_0 + t(\Lc_0-\Kc_{10}), \\
\Lc_1 \!\! &  =\  2t \Lc_\varepsilon \Lc_1 + t\Lc_\varepsilon \Kc_\varepsilon + t\Lc_1 + t(\Lc_1-\Kc_{01}), 
\end{cases}
\eeq
with $\Kc_\varepsilon =1$.
Using the 0/1 symmetry, this system can be compacted into
$$
\begin{cases}
\Kc_{01} \!\!\!\!& = \  t + t\Lc_0,
\\
\Lc_\varepsilon \!\!\!\!& = \  1+ 2t \Lc_\vareps \Lc_\vareps+ t(\Lc_\vareps-1) + 2t(\Lc_0-K_{01}),
 \\
\Lc_0 \!\!\!\!& = \  2t \Lc_\varepsilon \Lc_0 + t\Lc_\varepsilon + t\Lc_0 + t(\Lc_0-\Kc_{01}).
\end{cases}
$$
The variant mentioned in Remark~\ref{forward} consists in replacing
the {second} equation by $\Lc_\varepsilon=1+ 2\Lc_0$.  The reader may check that this is
consistent with the above system.

Eliminating $\Lc_0$ and $\Kc_{01}$ gives a quadratic
equation for the \gf\ $\Lc_\vareps= \Lc_\vareps^{(1)}$ of Eulerian
orientations in $\mathcal L^{(1)}$:
\begin{equation} \label{eq:eq_inf_1}
2t\Lc_\varepsilon^2-\Lc_\varepsilon(1-t)^2-t^2-2t+1=0.
\end{equation}
We defer to Section~\ref{sec:asympt-c} the study of the asymptotic behaviour of
its coefficients.

\subsubsection{When $\boldsymbol{k=2}$,}
the system $S_0$ 
contains $\e(2)=21$ equations, or 11 if we exploit the 0/1 symmetry:
\beq\label{eq:syst_inf_2}
\begin{cases}
\Kc_{10} =\Kc_{01}\!\!\!\! & = \  t + t\Lc_0, \\
\Kc_{1100} \!\!\!\! & = \  t\Kc_{10}  + t \Lc_{100}, \\
\Kc_{1010} \!\!\!\! & = \  t( \Kc_{10} + \Kc_{01} ) + t \Lc_{010}, \\
\Kc_{0110} \!\!\!\! & = \  t \Kc_{10} + t \Lc_{110} ,
\\
\Le \!\!\!\! & = \  1+2t\Lc_\vareps \Lc_\vareps + t(\Lc_\vareps-1) + 2t (\Lc_0-K_{10}
+\Lc_{100}-K_{1100}
\\& \hskip 70mm 
+ \Lc_{010}-K_{1010} + \Lc_{110}-K_{0110}),
\\
\Lc_0 = \Lc_1\!\!\!\! & = \   2t \Lc_\vareps \Lc_0 + t\Lc_\vareps + t \Lc_0 + t(\Lc_0-K_{10}
+\Lc_{100}-K_{1100}
\\ &\hskip 70mm 
+ \Lc_{010}-K_{1010} + \Lc_{110}-K_{0110}),
\\
\Lc_{00} = \Lc_{11}\!\!\!\!& = \  2t\Le \Lc_{00} + t\Lc_0  + t\Lc_{00} + t(\Lc_{100}-\Kc_{1100}), \\
\Lc_{10} = \Lc_{01}\!\!\!\!& = \  2t \Lc_\vareps \Lc_{10} + t \Lc_1 + t + t\Lc_{10} +
t (\Lc_0-K_{10} +\Lc_{010}-K_{1010} + \Lc_{110}-K_{0110}),
\\
\Lc_{100} \!\!\!\! & = \  2t\Le \Lc_{100} + t\Lc_{10}  + t \Lc_{100} + t(\Lc_{100}-\Kc_{1100}) ,\\
\Lc_{010} \!\!\!\! & = \  2t\Le \Lc_{010} + t(\Lc_{01}  + \Lc_\vareps \Kc_{10}) + t\Lc_{010} +
t(\Lc_{010}-\Kc_{1010}) , \\
\Lc_{110} \!\!\!\! & = \  2t\Le \Lc_{110} + t(\Lc_{11}  + \Lc_\vareps \Kc_{10}) + t \Lc_{110} + t(\Lc_{110}-\Kc_{0110}).
\end{cases}
\eeq
The variant mentioned in Remark~\ref{forward} consists in replacing
the  equations defining  $\Lc_\varepsilon$, $\Lc_0$ and $\Lc_{10}$ by 
$\Lc_\vareps=1+ 2\Lc_0$, $ \Lc_0=\Lc_{00}+ \Lc_{10}$ and $\Lc_{10}= \Kc_{10}+
\Lc_{010} + \Lc_{110}$ respectively. 

Eliminating all series but $\Lc_\vareps$ gives a quartic
equation for the \gf\ $\Lc_\vareps= \Lc_\vareps^{(2)}$ of Eulerian
orientations in $\mathcal L^{(2)}$:
 \begin{multline}\label{eq:eq_inf_2}
8  t^3\Lc_\vareps^4-4 t^2 (3 t^3+4 t^2-6 t+3) \Lc_\vareps^3+2 t (3 t^5-12 t^4-10 t^3+14 t^2-10 t+3) \Lc_\vareps^2\\+(t-1) (11 t^5-10 t^4-6 t^3-3 t^2-t+1) \Lc_\vareps+(t-1) (5 t^5-4 t^4+6 t^3-7 t^2+5 t-1)
=0.
 \end{multline}
We defer to Section~\ref{sec:asympt-c} the study of the asymptotic behaviour of
its  coefficients.

\subsection{Asymptotic analysis for subsets of Eulerian orientations}
\label{sec:asympt-c}
Here, we apply the theory of \emm positive irreducible polynomial
systems~\cite[Sec.~VII.6]{flajolet-sedgewick}, to prove the following asymptotic result. 

\begin{prop}\label{lem:prop_systm_prev}
For $k\ge 1$, let $\rho_k$ denote the radius of convergence of the
series $\Lc_\vareps^{(k)}$, which counts
orientations of $\peom$. Then $\rho_k$ is the only singularity of
$\Lc_\vareps^{(k)}$ of minimal modulus, and it is of the  square root
type: as $t$ tends to $\rho_k$ from below,
$$
\Lc_\vareps^{(k)}(t)= \alpha - \beta \sqrt{1-t/\rho_k} \left(1+o(1)\right)
$$
for non-zero constants $\alpha$ and $\beta$ depending on $k$. 

The number  $\ell^{(k)}_n$  of orientations of size
$n$ in $\peom$ satisfies, as $n$ tends to infinity:
$$
\ell^{(k)}_n \sim  c \la_k^{n} n^{-3/2},
$$
where $\la_k=1/\rho_k$ and $c= -\beta/\Gamma(-1/2)$.
\end{prop}

\begin{proof}
We use the terminology of~\cite[Sec.~VII.6.3]{flajolet-sedgewick}. Our
first objective is to transform the system $S_0$ of
Proposition~\ref{prop:l-c} into a positive one. The  obstructions 
to positivity come from the expression~\eqref{L-l-c2} of $\Lc_\gw$, and more
precisely from the terms 
$\Lc_\vareps-1$ (when $\gw=\vareps$) and 
$\Lc_{\gu_s}-\Kc_{{\gu}}$, where $\gu$ is balanced. These
  terms can be written  $\Lc_\gw -
\Kc_{\overleftarrow{\gw}}$, where $\gw=\gu_s$ is quasi-balanced  and 
$\overleftarrow{\gw}$ is the unique balanced word of the form $a\gw$,
for $a\in \{0,1\}$.

This leads us to define, for $\gw$  quasi-balanced  of length
less than $2k$, the series $\Lc^+_\gw := \Lc_\gw -
\Kc_{\overleftarrow{\gw}}$. We will also need to define, for $\gw$ balanced, $\Lc^+_\gw := \Lc_\gw -
\Kc_{{\gw}}$. These series have  natural combinatorial interpretations
in terms of orientations whose root word ends \emm strictly, with $\gw$ (if
$\gw$ is balanced) or $\overleftarrow{\gw}$ (if $\gw$ is
quasi-balanced). Then we alter the original system $S_0$ as
follows.
\begin{enumerate}
\item[(i)] For $\gw$ balanced or quasi-balanced, we replace the equation~\eqref{L-l-c2}  defining
  $\Lc_\gw$ by an equation defining $\Lc^+_\gw$:
\beq\label{L-l-c2-var}
  \Lc^+_\gw= 2 t \Lc_\varepsilon \Lc_\gw + 
	t\sum\limits_{\gw = \gu a \gv}(\Lc_{\gu}-\Kc_\gu)\Kc_{\gv} + 
	t \Lc^+_\gw + 
	t\sum\limits_
{\substack{\gu = \gv \gw , \gu \not = \gw\\ |\gu| \leq 2k-1 \\ \gu \text{ quasi-balanced}}}
\Lc^+_\gu.
  \eeq
To obtain it, either we get back to the explanation of~\eqref{L-l-c2} and
remove from its  right-hand side
 the terms  that count orientations with root word exactly $\gw$ (if
 $\gw$ is balanced) or ${\overleftarrow{\gw}}$ (if $\gw$ is
 quasi-balanced). Or we 
 simply subtract from~\eqref{L-l-c2} Equation~\eqref{F-l-c}, written for
 $\gw$ if $\gw$ is balanced, for $\overleftarrow{\gw}$ if $\gw$ is
 quasi-balanced.

\item [(ii)] In the new system thus obtained, we replace every
  series $\Kc_\gw$ such that $\gw$ is not balanced by~$0$, 
every
  series $\Lc_\gw$ such that $\gw$ is balanced by $\Kc_\gw+ \Lc^+_\gw$, and
  every series $\Lc_\gw$ such that $\gw$ is quasi-balanced by
  $\Kc_{\overleftarrow{\gw}}+ \Lc^+_\gw$.   In particular, the series
  $\Lc_\gu-K_\gu$ occurring in~\eqref{L-l-c2-var} becomes $\Lc_\gu^+$
  when $\gu$ is balanced, $\Lc_\gu$ otherwise. The
    only series $\Lc_\gw$ that remain in the system are such that the
    balance of $\gw$ is at least 2.
\end{enumerate}
We thus obtain a positive system, denoted $S_1$, defining the following series:
\begin{itemize}
\item $\Kc_\gw$, for
$\gw$ balanced of length between 2 and  $2k$, 
\item $\Lc^+_\gw$, for
$\gw$ balanced or quasi-balanced of length less than $2k$, 
\item $\Lc_\gw$, for $\gw$  valid  of length less than $2k$ and balance at least~2.
\end{itemize}
For instance, when $k=1$, the system~\eqref{syst-L-1} becomes (after
exploiting the 0/1 symmetry):
$$
\begin{cases}
\Kc_{01} \!\!\!\!& = \  t + t(K_{01}+\Lc_0^+),
\\
L^+_\varepsilon \!\!\!\!& = \   2t (1+L^+_\vareps )^2+ tL^+_\vareps + 2tL^+_0,
 \\
\Lc_0^+ \!\!\!\!& = \  2t  (1+L^+_\vareps ) (K_{01}+\Lc_0^+) + tL^+_\varepsilon + t\Lc_0^+ .
\end{cases}
$$
Similarly, when $k=2$, the system~\eqref{eq:syst_inf_2} is replaced by:
$$
\begin{cases}
\Kc_{10} =\Kc_{01}\!\!\!\! & = \  t + t(\Kc_{10}+\Lc^+_0), \\
\Kc_{1100} \!\!\!\! & = \  t\Kc_{10}  + t (\Kc_{1100}+\Lc^+_{100}), \\
\Kc_{1010} \!\!\!\! & = \  t( \Kc_{10} + \Kc_{01} ) + t (\Kc_{1010}+\Lc^+_{010}), \\
\Kc_{0110} \!\!\!\! & = \  t \Kc_{01} + t (\Kc_{0110}+\Lc^+_{110}) ,
\\
\Le^+ \!\!\!\! & = \   2t(1+\Le^+)^2+t\Le^++2t(\Lc_0^++\Lc_{100}^++\Lc_{010}^++\Lc_{110}^+),\\
\Lc^+_0 = \Lc^+_1\!\!\!\! & = \   2t(1+\Le^+)(K_{10}+\Lc_0^+)+t\Le^++t\Lc_0^++t(\Lc_{100}^++\Lc_{010}^++\Lc_{110}^+),
\\
\Lc_{00} = \Lc_{11}\!\!\!\!& = \  2t(1+\Le^+) \Lc_{00} + t(\Kc_{10} +\Lc^+_0 ) + t\Lc_{00} + t\Lc^+_{100} ,\\
\Lc^+_{10} = \Lc^+_{01}\!\!\!\!& = \  2t(1+\Le^+)(K_{10}+\Lc_{10}^+)+t(K_{01}+\Lc_1^+)+t\Lc_{10}^++t(\Lc_{110}^++\Lc_{010}^+),
\\
\Lc^+_{100} \!\!\!\! & = \  2t(1+\Le^+) (\Kc_{1100}+\Lc^+_{100}) + t\Lc^+_{10}  + t \Lc^+_{100}  ,\\
\Lc^+_{010} \!\!\!\! & = \  2t(1+\Le^+) (\Kc_{1010}+\Lc^+_{010}) + t(\Lc^+_{01}  + \Lc^+_\vareps \Kc_{10}) + t\Lc^+_{010}, \\
\Lc^+_{110} \!\!\!\! & = \  2t(1+\Le^+) (\Kc_{0110}+\Lc^+_{110}) + t(\Lc_{11}  + \Lc^+_\vareps \Kc_{10}) + t \Lc^+_{110}.
\end{cases}
$$

\medskip
The second condition that  we need is \emm properness, (again, in
the sense of~\cite[Sec.~VII.6.3]{flajolet-sedgewick}). But  the system
$S_1$ that we have just obtained is  proper, thanks to the factor $t$
occurring in the right-hand side of~\eqref{F-l-c},~\eqref{L-l-c2} and~\eqref{L-l-c2-var}.

\medskip
The next condition is aperiodicity. The coefficients of $t^1$ and
$t^2$ in the series $\Le^+(t)$ are both non-zero. This implies that
this series is aperiodic. Consequently, if we prove that the system
$S_1$ is
\emm irreducible, (which will be our final objective below), then it
will be aperiodic~\cite[p.~483]{flajolet-sedgewick}.

\medskip
So let us finally prove that $S_1$ is irreducible. Recall
that in such a polynomial system, a series $F$ \emm depends, on a series $G$ if $G$ occurs in the
right-hand side of the equation defining $F$. Irreducibility means
that the digraph of dependences is strongly connected. Recall that $S_1$ involves two families
of series: the series $\Kc_\gw$, for $\gw$  balanced of length between
2 and $2k$, and $\Lc_\gw$
(or $\Lc^+_\gw$) for any valid $\gw$ of length at most $2k-1$. To
lighten notation, for any  $\gw$ we denote by $\tLc_\gw$ the
corresponding $L$-series, be it  $\Lc_\gw$ or $\Lc^+_\gw$. 

Let us first prove that every series in $S_1$ depends on
$\tLc_\vareps=\Le^+$. By~\eqref{L-l-c2} and~\eqref{L-l-c2-var}, this holds for every $\tLc_\gw$. Now, each $\Kc_\gw$
 depends on at least one $\tLc$-series (see~\eqref{F-l-c}), and thus by
 transitivity on $\Le^+$.

Conversely, let us prove that $\Le^+$ depends on all other
series occurring in $S_1$.
\begin{itemize}
\item 
First,   Equation~\eqref{L-l-c2-var} applied to $\gw=\vareps$ shows that $\Le^+$ depends on all series $\tLc_\gu$ such that 
$\gu$ is quasi-balanced.
\item Let us now prove, by induction on the balance $b(\gu)$, that
  $\Le^+$ depends on $\tLc_\gu$ for each valid word~$\gu$ of length at
  most $2k-1$. We have already seen  this for
  $b(\gu)=1$. If $b(\gu)=0$, then $|\gu|\le 2k-2$, and for any letter $a$ the word
  $\gw:=\gu a$ is valid and quasi-balanced. The second term
  in~\eqref{L-l-c2-var} shows 
  that $\tLc_\gw$ depends on $\tLc_\gu$. By transitivity, this implies
  that $\Le^+$ depends on $\tLc_\gu$. We have thus set the initial
  cases of our induction, for balances~0 and 1. Now assume $b(\gu)\ge 2$. There exists a
  letter $a$ such that $\gw:=\gu a$ is valid and has balance
  $b(\gu)-1$.  If $b(\gw)=1$ (resp. $b(\gw)>1$), the second (resp. 
 third) term in the equation~\eqref{L-l-c2-var}
  (resp.~\eqref{L-l-c2}) defining $\tLc_\gw$ shows
  that $\tLc_\gw$ depends on $\tLc_\gu$. By the induction hypothesis,
  $\Le^+$ depends on $\tLc_\gw$, and thus by transitivity on $\tLc_\gu$.
\item Finally, let $\gu$ be balanced of length between 2 and
  $2k$. Then $\gw:=\gu_s$ is quasi-balanced.  The
  first term of the equation~\eqref{L-l-c2-var}
  defining $\tLc_\gw$ involves $\Lc_\gw= \Kc_\gu+ \tLc_\gw$, 
  so that by transitivity, $\Le^+$ depends on $\Kc_\gu$.
\end{itemize}

We have now checked all conditions of Theorem~VII.6
of~\cite[p.~489]{flajolet-sedgewick}. Applying it gives our proposition.
\end{proof}

\subsection{Back to examples}
\label{sec:ex-lc}
We now return to the cases $k=1$ and $k=2$ studied in
Section~\ref{sec:ex-c}. We refer
to~\cite[Sec.~VII.7]{flajolet-sedgewick} for generalities on the
singularities of algebraic series, and on the asymptotic behaviour of
their coefficients. When $k=1$, we have obtained for $\Le$ the
quadratic equation~\eqref{eq:eq_inf_1}. Its dominant coefficient 
only vanishes at $t=0$, and its discriminant is 
 $\Delta_1(t):=t^4+4t^3+22t^2-12t+1$. The radius $\rho_1$ must be one
 of the roots of $\Delta_1$.
The only real positive roots are around $0.1032$ and
$0.3998$. By solving~\eqref{eq:eq_inf_1}  explicitly, we see
that the smallest of these roots is indeed a singularity of
$\Le$. Hence $\rho_1=0.1032\ldots$ and the corresponding growth
rate is $\la_1=1/\rho_1= 9.684\ldots$, which improves on
the lower bound $8$ coming from Eulerian \emm maps,.

\medskip
When $k=2$, we have obtained for $\Le$ the
quartic equation~\eqref{eq:eq_inf_2}. Its dominant coefficient does
not vanish away from $0$, and its discriminant is 
\begin{multline*} \label{eq:disc_inf_2}
\Delta_2(t):=64t^{12}(t-1)(81t^{21} + 1863t^{20} + 11322t^{19} + 38592t^{18} +
101105t^{17} + 226631t^{16} + \\ 393423t^{15} + 532907t^{14} + 
665167t^{13} + 719797t^{12} + 454804t^{11} + 355710t^{10} + 
360159t^9 - 262135t^8 -  \\ 239969t^7 + 723151t^6 - 1106764t^5 + 
820832 t^4 - 316644t^3 + 65424t^2 - 6780t + 268).
\end{multline*}
The only  roots in $(0,1)$ are  $0.0984\ldots$ and $0.2714\ldots$. The
radius $\rho_2$ must be the first one (the other would give a growth
rate smaller than 8). Hence the corresponding growth
rate is $\la_2=1/\rho_2= 10.16\ldots$, which  improves on the previous
bound $\la_1$.

We do not push our study to larger values of $k$, as we will obtain
better bounds with the prime decomposition in the next section.

\section{Subsets of Eulerian orientations, via the  prime decomposition}
\label{sec:prime-lower}

In this section, we combine the restriction on allowed splits of the
previous section with the prime decomposition  of
Section~\ref{sec:prime-dec} to obtain a new family of subsets of
Eulerian orientations. The results and proofs are
similar to those of the previous section, and we give fewer details.
The new subsets $\pinf$ satisfy $\peom\subset \pinf$ (Proposition~\ref{prop:incl-l}), hence they give
better  lower bounds 
on the growth rate $\mu$  than those obtained in the previous
section.  Moreover these bounds  increase to  $\mu$ as $k$
increases (Proposition~\ref{prop:lc-incr}).

\medskip
Recall  from Section~\ref{sec:prime-dec} that an Eulerian orientation
is a sequence of {prime} Eulerian orientations, and that a prime (Eulerian) orientation 
 can be  obtained recursively from the atomic map by either:
\begin{itemize} 
\item adding a loop,  oriented in either way, around an orientation $O_1$,
\item or performing a legal split on a prime orientation $O' \in \PEO$, followed by
  the concatenation of an arbitrary Eulerian orientation $O''$ at the new
  vertex created by the split (Figure~\ref{fig:rec-prime}).
\end{itemize}

\begin{definition} \label{def:lb-prime}
Let $k\ge 1$. Let $\pinf$ be the set of planar orientations obtained
recursively from the atomic map by either:
\begin{itemize} 
\item  concatenating a sequence of prime orientations of $\pinf$,
\item or adding a loop,  oriented in either way, around an orientation
  $O_1$ of $\pinf$, 
\item or performing a legal $i$-split on a prime orientation $O' \in
  \pinf$, with $i=\dv(O')$ or $i\le k$, followed by
  the concatenation of an arbitrary  orientation of $\pinf$ at the new
  vertex created by the split.
\end{itemize}
\end{definition}
Clearly, the sets $\pinf$ increase to the set $\PEO$ of all Eulerian
orientations as $k$ increases, hence their growth rates $\bar \la_k$
form a non-decreasing sequence of lower bounds on $\mu$. But we have
in this case a stronger result.

\begin{prop}\label{prop:lc-incr}
  For $k\ge 1$, the sequence $(\bar{\ell}_n^{(k)})_{n\ge 0}$ that counts
  orientations of $\pinf$ by their size is
  super-multiplicative. Consequently, the associated growth rate
\beq\label{super-mult-bound}
\bar \la_k:= \lim_n \ (\bar\ell_n^{(k)})^{1/n}=\sup_n\   (\bar\ell_n^{(k)})^{1/n}
\eeq
increases to $\mu$ as $k$ tends to infinity.
\end{prop}
\begin{proof}
  By definition of $\pinf$, concatenating two orientations of $\pinf$
  at their root vertex gives a new element of $\pinf$, which implies
  super-multiplicativity and the identity~\eqref{super-mult-bound}
  (by Fekete's Lemma~\cite[p.~103]{vanlint01}).

Now since $\pinf$ converges to $\PEO$, for any $n$, there exists $k$
such that
$o_n= \ell_n^{(k)}$ (one can take $k=n$, or
even $k=n-2$). Hence
$$
o_n^{1/n} = (\bar\ell_n^{(k)})^{1/n} \le \bar\la_k \le \lim_k \bar\la_k,
$$
and it follows now from~\eqref{fekete} that $\mu\le \lim \bar \la _k$. Since
$\bar \la_k\le \mu$, the proposition follows.
\end{proof}

\begin{prop}\label{prop:incl-l}
For $k\ge 1$, the subset of orientations $\pinf$ includes the subset
$\peom$ defined in Section~\rm{\ref{sec:classic-lower}}. 
\end{prop}
\begin{proof}
  We prove this by induction on the number of edges. The inclusion is
  obvious for orientations with no edge. Now let $O\in \peom$, having
  at least one edge. 

If $O$ is the merge of two orientations $O_1$ and
  $O_2$, then the induction hypothesis implies that $O_1$ and $O_2$
  are in $\pinf$. The structure of $\pinf$ implies that every prime
  sub-orientation of $O_2$ (attached at the root of $O_2$) also belongs to
  $\pinf$. Then $O$ can be obtained as an orientation of $\pinf$ by
  first adding a loop around $O_1$ (this is the second construction in
  Definition~\ref{def:lb-prime}), then concatenating one by one the prime sub-orientations
  of $O_2$ (first construction in Definition~\ref{def:lb-prime}).

Otherwise, $O$ is obtained by a legal split in an orientation $O'$
formed of the prime sub-orientations $P_1, \ldots, P_\ell$. By the induction
hypothesis, $O'$, and its prime sub-orientations $P_1, \ldots, P_\ell$, belong
to $\pinf$. Let us say that the
split occurs in $P_i$ (this means that the sub-orientations $P_1, \ldots,
P_{i-1}$ are attached to the new created vertex $v'$, while $P_{i+1},
\ldots, P_\ell$ remain attached to the original vertex $v$, the root vertex of $O$). Then the
orientation $O_1$ obtained by deleting from $O$ the sub-orientations  $P_{i+1},
\ldots, P_\ell$ can be obtained by a legal split in the prime
orientation $P_i$, followed by the concatenation of  $P_1, \ldots,
P_{i-1}$ at the new created vertex. This is the third construction in
Definition~\ref{def:lb-prime}, hence $O_1$ belongs to $\pinf$. It
remains to concatenate  $P_{i+1},
\ldots, P_\ell$ at the root (first construction in
Definition~\ref{def:lb-prime}), and we recover $O$ as an element of~$\pinf$.
\end{proof}

\subsection{An algebraic system for $\boldsymbol\pinf$}
We now fix $k\ge 1$. For $\gw$ a word on $\{0,1\}$, let
$\Lp_\gw^{(k)}(t) \equiv \Lp_\gw$ denote the \gf\ of
orientations of $\pinf$ whose root word ends with $\gw$. Let $\Kp_\gw$
be the \gf\ of those that have root word exactly $\gw$. Let
$\Lpp_\gw$ and $\Kpp_\gw$ be the corresponding series for
\emm prime, orientations of $\pinf$. We are especially interested in
the series $\Lp_\vareps$ that counts all orientations of $\pinf$.

 If $\gw$ is not balanced, $\Kp_\gw=\Kpp_\gw=0$, while if
$\gw=\vareps$, $\Kp_\gw=1$ and $\Kpp_\gw=0$.  For $\gw$ non-empty
and balanced, of length at most $2k$, we have
\beq\label{F-l-p}
\Kp_\gw= \sum\limits_{\gw= \gu \gv} \Kp_{\gu} \Kpp_{\gv},
\eeq
since an orientation of $\pinf$ is a sequence of orientations of
$\pinf$. Now the description of prime orientations of $\pinf$
(Definition~\ref{def:lb-prime}) gives
\beq\label{Fp-l-p}
\Kpp_\gw=t \Kp_{\gw_c} + t \Lpe \Lpp_{\gw_s}.
\eeq
The first term corresponds to adding a loop, and the second to
a legal $i$-split, where $i\le k$ is the half-length of $\gw$. The
factor $\Lp_\vareps$ accounts for the orientation $O''$ attached at the end
of the root edge.

Now let $\gw$ be a valid word of length at most $2k-1$, and let us
write equations for the series $\Lp_\gw$ and $\Lpp_\gw$. For $\Lp_\gw$,
the sequential structure of orientations of $\pinf$ gives
\beq\label{L-l-p}
\Lp_\gw  = \mathbbm{1}_{\gw=\vareps}+ \Lpe \Lpp_\gw + \sum\limits_{\gw= \gu \gv, \gv \neq  \gw} \Lp_{\gu} \Kpp_{\gv}.
\eeq
The second (resp.~third) term counts orientations in which the root word
  of the last prime component
 ends with $\gw$ (resp.~is shorter than $\gw$). Finally, for the series
$\Lpp_\gw$ we obtain the following  counterpart
of~\eqref{L-l-c2}:
\begin{multline}\label{Lpp-l-p}
\Lpp_\gw  = 2t \Lp_\vareps \mathbbm{1}_{\gw=\vareps}+ t
\Lp_{\gw_p}\mathbbm{1}_{\gw\not = \vareps }
+t \Kp_{\gw_c}\mathbbm{1}_{\gw\not = \vareps   \mbox{ balanced}}
+ t \Lp_\vareps\left(
 \Lpp_\gw 
+	\sum\limits_
{\substack{\gu = \gv \gw \\ 0<|\gu| \le 2k \\ \gu \text{ balanced}}}
 (\Lpp_{\gu_s} - \Kpp_{{\gu}})\right) .
\end{multline}
The first three terms count
orientations in which the root edge is a loop, and the last one
those obtained by a split.

\begin{prop}\label{def-syst-l-p}
  Consider the collection of equations consisting of:
  \begin{itemize}
  \item Equation~\eqref{F-l-p}, written for all balanced words $\gw$
    of length between $2$ and $2k-2$,
\item Equation~\eqref{Fp-l-p}, written for all balanced words $\gw$
    of length between $2$ and $2k$,
\item Equation~\eqref{L-l-p}, written for all valid words $\gw$
	of length at most $2k-2$,
 \item Equation~\eqref{Lpp-l-p}, written for all valid words $\gw$
	of length at most $2k-1$.
  \end{itemize}
In this collection, replace all trivial $\Kp$- and $\Kpp$-series by their value:
$\Kp_\gw=\Kpp_\gw=0$ when $\gw$ is not balanced, $\Kp_\vareps=1$,
$\Kpp_\vareps=0$. Let $\sf S_0$
denote the resulting system. The number of
series it involves is $2\e(k)-2{2k \choose k}$, where $\e(k)$ is given
by~\eqref{ek-def}. Moreover, $\sf S_0$ defines uniquely all these 
series. Its size  can be (roughly) divided by two upon
exploiting the $0/1$ symmetry.
\end{prop}
\begin{proof}
  To prove that all series are well defined by the system, we first
  check that every series occurring in the right-hand side of some
  equation is the left-hand side of another equation. Then we note
  that:
  \begin{itemize}
  \item the   equations for prime orientations, namely~\eqref{Fp-l-p}
  and~\eqref{Lpp-l-p}, have a factor $t$ in their right-hand sides,
\item for the other two
  equations,~\eqref{F-l-p} and~\eqref{L-l-p}, every non-trivial term
  in the right-hand side has a series  of prime
  orientations as a factor.  
\end{itemize}
Now the number of equations: every series that was occurring in the
system $S_0$ of Proposition~\ref{prop:l-c} now has two copies (one with a
prime, one without), except for the series $\Kc_\gw$, for $\gw$
balanced of length $2k$, and $\Lc_\gw$, for $\gw$ quasi-balanced of
length $2k-1$, which have only one copy. Since there are ${2k \choose
  k}$ balanced words of length $2k$, and $2{2k-1 \choose
  k}= {2k \choose
  k}$ quasi-balanced words of length $2k-1$, the result follows.
\end{proof}

\begin{remark}
 As in Remark~\ref{forward}, if $\gw$ is such that $0\gw$ and $1\gw$ are
both valid of length less than $2k-2$ (resp. $2k-1$), we can
replace~\eqref{L-l-p} (resp.~\eqref{Lpp-l-p}) by the simpler forward equation:
$$
\Lp_\gw= \Kp_\gw + \Lp_{0\gw}+ \Lp_{1\gw} \qquad 
(\hbox{resp. }  \Lpp_\gw= \Kpp_\gw + \Lpp_{0\gw}+ \Lpp_{1\gw} ).
$$
This does not increase the size of the system.
\end{remark}

\subsection{Examples}\label{sec:ex-p}

\subsubsection{When  $\boldsymbol{k=1}$,}
the system $\sf S_0$ 
contains $2(\e(1)-2)=6$ equations, or 4 of we exploit the $0/1$ symmetry:
$$
\begin{cases}
\Kpp_{10} \!\!\!\!& = \  t+t\Lp_\varepsilon \Lpp_0,\\
\Lp_\vareps\!\!\!\!& = \ 1+ \Lp_\vareps\Lpp_\vareps, \\
\Lpp_\varepsilon \!\!\!\!& = \   2t \Lp_\vareps+  t \Lp_\vareps(\Lpp_\vareps+2\Lpp_0 - 2\Kpp_{10}),
\\
\Lpp_0 \!\!\!\!& = \  t \Lp_\varepsilon + t\Lp_\varepsilon (\Lpp_0+\Lpp_0 - \Kpp_{10}).
\end{cases}
$$
Eliminating all series but $\Lp_\vareps$  gives a cubic
equation for the \gf\ $\Lp_\vareps\equiv\Lp_\vareps^{(1)}$ of Eulerian
orientations in $\mathbb L^{(1)}$:
\begin{equation} \label{eq:eq_inf_1-p}
t^2\Lp_\vareps^3 +t(t-4)\Lp_\varepsilon^2+(2t+1)\Lp_\varepsilon -1=0.
\end{equation}

\subsubsection{When  $\boldsymbol{k=2}$,}
the system $\sf S_0$ 
contains $2(\e(2)-6)=30$ equations, or 16 if we exploit the 0/1 symmetry:
\beq
\begin{cases}\label{eq:syst_inf_2-p}
\Kp_{01}=\Kp_{10}\!\!\!\!& = \  \Kpp_{01},
\\
\Kpp_{10} =\Kpp_{01}\!\!\!\! & = \  t + t\Lp_\vareps\Lpp_1, 
\\
\Kpp_{1100} \!\!\!\! & = \  t\Kp_{10}  + t\Lp_\vareps \Lpp_{100}, \\
\Kpp_{1010} \!\!\!\! & = \  t \Kp_{01}  + t\Lp_\vareps \Lpp_{010}, \\
\Kpp_{0110} \!\!\!\! & = \   t\Lp_\vareps \Lpp_{110} ,
\\
\Lp_\vareps \!\!\!\! & = \ 1+\Lp_\vareps\Lpp_\vareps,
\\
\Lp_0=\Lp_1\!\!\!\!& = \  \Lp_\vareps \Lpp_{0},
\\
\Lp_{00} = \Lp_{11}\!\!\!\!& = \  \Lp_\vareps\Lpp_{00},
\\
\Lp_{01} =\Lp_{10}\!\!\!\!& = \  \Lp_\vareps\Lpp_{01},
\\
\Lpp_\vareps \!\!\!\! & = \  2t\Lp_\vareps  + t\Lp_\vareps(\Lpp_\vareps+ 2(\Lpp_0-\Kpp_{10}
+\Lpp_{100}-\Kpp_{1100}+ \Lpp_{010}-\Kpp_{1010} + \Lpp_{110}-\Kpp_{0110})),
\\
\Lpp_0 = \Lpp_1\!\!\!\! & = \   t \Lp_\vareps + t\Lp_\vareps(\Lpp_0 + \Lpp_0 -\Kpp_{10}
+\Lpp_{100}-\Kpp_{1100}+ \Lpp_{010}-\Kpp_{1010} + \Lpp_{110}-\Kpp_{0110}),
\\
\Lpp_{00}\!\!\!\!& = \  t\Lp_0+t\Lp_\vareps( \Lpp_{00}+ \Lpp_{100}-\Kpp_{1100}),
\\
\Lpp_{10} = \Lpp_{01}\!\!\!\!& = \  t \Lp_1+t+ t \Lp_\vareps(\Lpp_{10} +\Lpp_{1}
-\Kpp_{01} 
+ \Lpp_{010}-\Kpp_{1010} + \Lpp_{110}-\Kpp_  {0110}),
\\
\Lpp_{100} \!\!\!\! & = \  t \Lp_{10}  + t\Lp_\vareps(\Lpp_{100}+\Lpp_{100}-\Kpp_{1100}) ,\\
\Lpp_{010} \!\!\!\! & = \  t\Lp_{01} + t\Lp_\vareps(\Lpp_{010}+\Lpp_{010}-\Kpp_{1010}) , \\
\Lpp_{110} \!\!\!\! & = \  t \Lp_{11}+ t\Lp_\vareps(\Lpp_{110}+\Lpp_{110}-\Kpp_{0110}).\\
\end{cases}
\eeq
Eliminating all series but $\Lp_\vareps$ gives an
equation of degree 6 for the \gf\ $\Lp_\vareps\equiv\Lp_\vareps^{(2)}$ of Eulerian
orientations in $\mathbb L^{(2)}$:
 \begin{multline}\label{eq:eq_inf_2-p}
2t^5\Lp_\vareps ^6-t^4(t+8)\Lp_\vareps ^5-t^3(3t^2-16)\Lp_\vareps ^4+t^2(2t+3)(2t-5)\Lp_\vareps ^3-t(2t^2-7t-7)\Lp_\vareps ^2-(5t+1)\Lp_\vareps +1=0.
 \end{multline}

\subsection{Asymptotic analysis for subsets of Eulerian orientations   (prime decomposition)}
\label{sec:asympt-p}
We now prove for the polynomial system of
Proposition~\ref{def-syst-l-p} an analogue of
Proposition~\ref{lem:prop_systm_prev}.  

\begin{prop}
For $k\ge 1$, let $\bar\rho_k$ denote the radius of convergence of the
series $\Lp_\vareps^{(k)}$ that counts
orientations of $\pinf$. Then $\bar\rho_k$ is the only singularity of
$\Lp_\vareps^{(k)}$ of minimal modulus, and it is of the  square root
type. Consequently, there exists a constant $c$ such
that the number  $\bar\ell^{(k)}_n$  of orientations of size
$n$ in $\pinf$ satisfies, as $n$ tends to infinity:
$$
\bar\ell^{(k)}_n \sim  c \bar \lambda_k^{n} n^{-3/2},
$$
with $ \bar \lambda_k=1/\bar\rho_k$.
\end{prop}
\begin{proof}
  Again, we apply the theory of positive irreducible polynomial
systems~\cite[Sec.~VII.6]{flajolet-sedgewick}. 

The system of Proposition~\ref{def-syst-l-p} is not positive. To correct this, we
replace the series $\Lp_\gw$ (for $\gw$
balanced) and  $\Lpp_\gw$ (for $\gw$
balanced or quasi-balanced) by their ``positive''
versions:
$$
\Lp_\gw^+:= \Lp_\gw- \Kp_\gw, \qquad \Lppp{\gw}:=\Lpp_\gw- \Kpp_\gw \qquad
(\gw \hbox{ balanced)},
$$
$$
 \Lppp{\gw}:=\Lpp_\gw- \Kpp_{\overleftarrow\gw} \qquad (\gw \hbox{ quasi-balanced)}.
$$
In particular, $\Lp_\vareps$ is replaced by
$\Lp_\vareps^+:=\Lp_\vareps-1$ and $\Lppp{\vareps}$  coincides
with $\Lpp_\vareps$.  We alter the original system $\sf S_0$
as follows:
\begin{enumerate}
\item [(i)] For $\gw$ balanced, we replace the equation~\eqref{L-l-p}  defining
  $\Lp_\gw$ by the difference between~\eqref{L-l-p} and~\eqref{F-l-p}:
\beq\label{Lp-plus}
\Lp_\gw^+= \Lp_\vareps^+ \Lpp_\gw + \Lppp{\gw} +
\sum_{\gw=\gu \gv, \gv \not= \gw} \Lp_\gu \Kpp_\gv.
\eeq
\item [(ii)] For $\gw$ balanced or quasi-balanced, we replace the equation~\eqref{Lpp-l-p}  defining
  $\Lpp_\gw$ by the difference between~\eqref{Lpp-l-p}
  and~\eqref{Fp-l-p}  (written for $\gw$ if $\gw$ is balanced, for $\overleftarrow{\gw}$ otherwise):
\beq\label{Lpp-plus}
\Lppp{\gw}  = 2t \Lp_\vareps \mathbbm{1}_{\gw=\vareps}
+ t (\Lp_{\gw_p}-\Kp_{\gw_p}) \mathbbm{1}_{\gw\not = \vareps}
+ t \Lp_\vareps\left( \Lppp{\gw }
+	\sum\limits_
{\substack{\gu = \gv \gw, \gu\not = \gw \\ |\gu| \leq 2k-1 \\ \gu \text{ quasi-balanced}}}
 {\Lppp{\gu}}\right) .
\eeq
\item [(iii)] In the new system thus obtained, we replace every
  series $\Kp_\gw$ such that $\gw$ is not balanced by $0$, 
every
  series $\Lp_\gw$ (resp. $\Lpp_\gw$) such that $\gw$ is balanced by
  $\Kp_\gw+ \Lp^+_\gw$ (resp. ${\Kpp_\gw}+ \Lppp{\gw}$), and
  every series $\Lpp_\gw$ such that $\gw$ is quasi-balanced by
  $\Kpp_{\overleftarrow{\gw}}+ \Lppp{\gw}$.   In particular, the series
  $\Lc_{\gw_p}-K_{\gw_p}$ occurring in~\eqref{Lpp-plus} becomes $\Lc_{\gw_p}^+$
  when $\gw_p$ is balanced, $\Lc_{\gw_p}$ otherwise. The
    series  $\Lp_\gw$ (resp. $\Lpp_\gw$) that remain in the system are
    such that $\gw$ has balance at least~1 (resp. 2).
\end{enumerate}
We thus obtain a positive system, denoted $\sf S_1$, defining the following series:
\begin{itemize}
\item $\Kp_\gw$, for
$\gw$ balanced of length between 2 and  $2k-2$, 
\item $\Kpp_\gw$, for
$\gw$ balanced of length between 2 and  $2k$, 
\item $\Lp_\gw$ for $\gw$ valid  of balance at least~1 and length at
  most $2k-2$,
\item $\Lpp_\gw$ for $\gw$ valid  of balance at least~2 and length at
  most $2k-1$,
\item $\Lp^+_\gw$, for
$\gw$ balanced  of length at most  $2k-2$, 
\item $\Lppp{\gw}$, for $\gw$  balanced or quasi-balanced of length
  at most $2k-1$.
\end{itemize}
For instance, when $k=1$ we obtain the following system:
$$
\begin{cases}
\Kpp_{10} \!\!\!\!& = \  t+t(1+\Lp_\varepsilon^+)(\Kpp_{10}+ \Lppp{0}),
\\
\Lp_\vareps^+\!\!\!\!& = \  \Lp_\vareps^+{\Lpp_\vareps}^++ \Lppp{\vareps}, \\
\Lppp{\varepsilon} \!\!\!\!& = \   2t (1+\Lp_\vareps^+)+ 
 t (1+\Lp_\vareps^+)({\Lpp_\vareps}^++2\Lppp{0}), 
\\
\Lppp{0} \!\!\!\!& = \  t \Lp_\varepsilon^+ + t(1+\Lp_\varepsilon^+) \Lppp{0}.
\end{cases}
$$
Recall that the series we are interested in is $\Lp_\vareps^+$. But
then we can drop  the first equation of the above system. \emm This
size reduction occurs for any value of $k$,, and the positive system
$\sf S_2$ that we
will study is finally obtained by performing one last change:
\begin{itemize}
\item[(iv)] Delete the equations defining the series $\Kpp_\gw$, for
  $\gw$ of length $2k$.
\end{itemize}
Observe that all the series involved in $\sf S_2$ are well-defined by
this system. This comes from the fact that all the series $\Kpp_\gu$, for
  $\gu$ of length $2k$, that occurred in $\sf S_0$ came from the
  term $\Lpp_{\gu_s}- \Kpp_{\gu}$ of~\eqref{Lpp-l-p}, which now reads
  ${\Lppp{\gu_s}}$.

 Here is for instance the system obtained for $k=2$, which has three
equations less than~\eqref{eq:syst_inf_2-p}:
$$
\begin{cases}
\Kp_{01}=\Kp_{10}\!\!\!\!& = \  \Kpp_{01},
\\
\Kpp_{10} =\Kpp_{01}\!\!\!\! & = \  t + t(1+\Lp_\vareps^+)(\Kpp_{01}+\Lppp{1}), 
\\
\Lp_\vareps ^+ \!\!\!\! & = \ \Lp_\vareps^+\Lppp{\vareps}+ \Lppp{\vareps},
\\
\Lp_0 \!\!\!\! & = \   (1+\Lp_\vareps^+) (\Kpp_{10}+{\Lppp{0}}) ,
\\
\Lp_{00} = \Lp_{11}\!\!\!\!& = \  (1+\Lp_\vareps^+)\Lpp_{00},
\\
\Lp_{01}^+ \!\!\!\!& = \  \Lp_\vareps ^+(\Kpp_{01}+{\Lppp{01}})+ {\Lppp{01}},
\\
\Lppp{\vareps}\!\!\!\! & = \  2t(1+\Lp_\vareps^+)  +
t(1+\Lp_\vareps^+)(\Lppp{\vareps}+ 2(\Lppp{0} 
+{\Lppp{100}}+ {\Lppp{010}} + {\Lppp{110}})),
\\
\Lppp{0} = \Lppp{1} \!\!\!\! & = \   t \Lp_\vareps^+ +
 t(1+\Lp_\vareps^+)(\Lppp{0} 
+{\Lppp{100}} + {\Lppp{010}} + {\Lppp{110}}),
\\
\Lpp_{00}\!\!\!\!& = \  t\Lp_0+t(1+\Lp_\vareps^+)( \Lpp_{00}+ {\Lppp{100}}),
\\
{\Lppp{10}} = {\Lppp{01}}\!\!\!\!& = \  t\Lp_1+ t (1+\Lp_\vareps^+)({\Lppp{10}}
+ {\Lppp{010}} + {\Lppp{110}}),
\\
{\Lppp{100}} \!\!\!\! & = \  t \Lp_{10}^+  + t(1+\Lp_\vareps^+){\Lppp{100}} ,\\
{\Lppp{010}} \!\!\!\! & = \  t\Lp_{01}^+ + t(1+\Lp_\vareps^+){\Lppp{010}} , \\
{\Lppp{110}} \!\!\!\! & = \  t \Lp_{11}+ t(1+\Lp_\vareps^+){\Lppp{110}}. \\
\end{cases}
$$

\medskip
Let us now discuss  properness~\cite[p.~489]{flajolet-sedgewick}. The
system $\sf S_2$ that we have just obtained
is not proper.  However, the right-hand sides of the equations that
define  series with a prime ($\Kpp, \Lpp$ and ${\Lpp}^+$)  are multiples of
$t$ (see~\eqref{Fp-l-p},~\eqref{Lpp-l-p} and~\eqref{Lpp-plus}). In the
remaining equations, that is~\eqref{F-l-p},~\eqref{L-l-p} (for $\gw$
not balanced) and~\eqref{Lp-plus} (for $\gw$ balanced),
each term on the right-hand side  involves a series with a prime:
hence after one iteration of $\sf S_2$, one obtains a new system $\sf
S_3$ which is positive and proper.

Aperiodicity holds as in the previous section, and we are left with
irreducibility. Note that proving irreducibility for $\sf S_2$ or its
iterated version $\sf S_3$ is equivalent, so we focus on $\sf S_2$. As in the previous section, we denote by $\tLp_\gw$
the series $\Lp_\gw^+$ or $\Lp_\gw$, depending on whether $\gw$ is
balanced or not. Similarly, $\tLpp_\gw$ denotes $\Lppp{\gw}$ if
$\gw$ is balanced or quasi-balanced, and $\Lpp_\gw$
otherwise. Let us prove that all series depend on
  $\tLp_\vareps$. We
first observe that this holds for every $\Kpp$- or $\tLp$- or $\tLpp$-series (see~\eqref{Fp-l-p}, \eqref{L-l-p}, \eqref{Lpp-l-p},
\eqref{Lp-plus}, \eqref{Lpp-plus}). We are left with  the series
$\Kp_\gw$: but
it depends on $\Kpp_\gw$ (see~\eqref{F-l-p}), and hence on $\tLp_\vareps$. 

Conversely, let us prove that  $\tLp_\vareps$ depends on every other
series in the system. By~\eqref{Lp-plus}, it depends on $\Lpp_\vareps$.
Then by~\eqref{Lpp-plus} applied to $\gw=\vareps$, it depends on every series
$\Lpp_\gu$, where $\gu$ is quasi-balanced. Going back and forth
between the equations defining the $\Lp$- series and the $
\Lpp$-series (see \eqref{L-l-p}, \eqref{Lpp-l-p}, \eqref{Lp-plus},
\eqref{Lpp-plus}), and using an induction on the balance, we then see
that  $\tLp_\vareps$ depends on all series $\tLp_\gu$
(for $|\gu|\le 2k-2$)  and all series $
\tLpp_\gu$ (for $|\gu|\le 2k-1$). Then the first  term
of~\eqref{Lp-plus}, written as $\Lp_\vareps^+(\Kpp_\gw+\Lppp{\gw})$,
shows that   $\tLp_\vareps$
depends on all series $\Kpp_\gv$ with $|\gv|\le 2k-2$.
It remains to prove that  $\tLp_\vareps$ depends on the
$\Kp$-series. Let $\gu=a\gu_s$ be balanced of length at most $2k-2$,
and 
define $\gw=\gu_s \bar a$. This word has balance 2. The second term
of~\eqref{Lpp-l-p} involves $\Lp_{\gw_p}= \Lp_{\gu_s}= \Kp_\gu+
\Lp_{\gu_s}^+$. Hence $\Lpp_\gw$ depends on $\Kp_\gu$, and by
transitivity, $\tLp_\vareps$ depends on $\Kp_\gu$. This proves the
irreducibility of the system and concludes the proof of the proposition.
\end{proof}

\subsection{Back to examples}
We first return to the cases $k=1$ and $k=2$ studied in
Section~\ref{sec:ex-p}. When $k=1$, we obtained the cubic
equation~\eqref{eq:eq_inf_1-p} for $\Lp_\vareps^{(1)}$.  The discriminant has three positive roots, which are 1, and  (approximately)
$0.094$ and $15.9$. The second one is the radius of convergence, and
we obtain the lower bound $\bar \lambda_1\simeq 10.603$ on the growth
rate of Eulerian orientations. This improves significantly on the
growth rate $\la_1=9.68\ldots$ obtained from the set $\mathcal
L^{(1)}$. 

For $k=2$, we obtained the equation~\eqref{eq:eq_inf_2-p} satisfied
by $\Lp_\vareps^{(2)}$. The discriminant has two roots
in $(0,1)$, which are approximately
$0.0911$ and $0.414$. The first one is the radius of convergence, and
we obtain the lower bound $\bar \lambda_2\simeq 10.9759$ on the growth
rate of Eulerian orientations.

\medskip
When $k=3$, we find that $\Lp_\vareps^{(3)}$ satisfies an equation of degree
20 (see the {\sc Maple} sessions available on our web pages). The dominant coefficient only vanishes at $t=8$, and the
discriminant has only one relevant root, around 0.089. This gives the
lower bound $\bar \lambda_3\simeq11.2289$ on the growth rate  of
Eulerian orientations.

\medskip
When $k=4$, we did not compute the equation satisfied by  $\Lp_\vareps^{(4)}$, but
we estimated $\bar\lambda_4$ from the first 30 coefficients
of $\Lp_\vareps^{(4)}$ using quadratic
approximants~\cite{brak-guttmann}. We predict $\bar \lambda_4\simeq
11.41$. This value has then been confirmed by Bruno Salvy 
using the Maple package NewtonGF~\cite{pivoteau}, with which he obtained 10 digits of~$\bar\lambda_4$. This package also allows us to compute more
coefficients in~$\Lp_\vareps^{(4)}$. Moreover, Jean-Charles Faugère~\cite{faugere} has finally been able to
determine the equation for $\Lp_\vareps^{(4)}$, which has degree 258
in~$\Lp_\vareps^{(4)}$. 

 Similarly, we predict
$$
\bar\lambda_5\simeq 11.56, \qquad \bar\lambda_6\simeq11.68.
$$

\section{Supersets of Eulerian orientations, via the standard decomposition}
\label{sec:classic-upper}
We now want to define, and count,  supersets of Eulerian
orientations. Their \gfs\ will be described by functional equations
involving divided differences (as in~\eqref{eq-func-M}). The proof of their
algebraicity is non-trivial, relying on a deep result from Artin's approximation
theory (Theorem~\ref{thm:popescu}).

\medskip
Recall that  Eulerian orientations
can be  obtained recursively from the atomic map by either:
\begin{itemize} 
\item the merge of two orientations $O_1, O_2 \in \PEO$ (with the root
  loop oriented in either way),
\item or a legal split on an orientation $O' \in \PEO$.
\end{itemize}

We now define the sets $\peop$. The idea is that we allow illegal
$i$-splits, provided $i$ is larger than $k$.
\begin{definition} \label{def:ub-classic}
Let $k\ge 1$. 
Let $\peop$ be the set of planar orientations obtained recursively
from the atomic map by either:
\begin{itemize} 
\item the merge of two orientations $O_1, O_2 \in \peop$  (with the root
  loop oriented in either way),
\item or a legal $i$-split on a map $O' \in \peop$ with $i\leq k$ (small split), 
\item or an arbitrary split on a map $O' \in \peop$ with  $ i>k $ (large split). If the split is legal,  the root edge is oriented in the
  only way that makes the new orientation Eulerian.  Otherwise, it is oriented
  away from the root vertex.
\end{itemize}
\end{definition}

Observe that all
Eulerian orientations belong to $\peop$. Moreover,  the sets $\peop$
form a decreasing 
sequence, as fewer illegal splits are performed as $k$ grows.
Finally, for $k\ge n$ (and even for $k\ge n-2$), all orientations of size $n$ in $\peop$ are
  Eulerian. Hence the
limit of the sets $\peop$ is the set $\PEO$ of
all Eulerian orientations.

Another  important observation is that, if the root vertex  of an
orientation of $\peop$ has degree at most $2k$, then the root word of this
orientation is balanced. 

\subsection{Functional equations for $\boldsymbol \peop$}
We now fix an integer $k$.  For a word $\gw$ on $\{0,1\}$, let
$\Ucl_\gw^{(k)}(t;x) \equiv \Ucl_\gw(x)$ denote the generating function of
orientations of $\peop$ whose root word ends with $\gw$, counted by
the edge number (variable $t$) and the half-degree of the root vertex
(variable $x$). Let  $\Tc_\gw^{(k)}(t) \equiv \Tc_\gw$ denote the
generating function of orientations of $\peop$ having root word
exactly $\gw$. We do not record in this series the root
degree (which is the length of $\gw$). To lighten notation, we often denote simply by
$\Ucun{\gw}$ the edge \gf\ $\Ucl_\gw(1)\equiv \Ucl_\gw(t,1)$, and by
$\Uc{\gw}$ the refined \gf\ $\Ucl_\gw(x)\equiv \Ucl_\gw(t;x)$.

Note that $\Tc_\gw=0$ if $\gw$ is not balanced and that
$\Tc_\vareps=1$. Now, for $\gw$ balanced of length between $2$ and
$2k$, we have 
\beq\label{Tw-c}
\Tc_\gw=
	t \sum\limits_{a \gu \bar{a} \gv = \gw} \Tc_{\gu} \Tc_{\gv}
	 +t  \Ucun{\gw_s} .
   \eeq
The first term counts orientations obtained by a merge. The second
one counts those obtained by a split, which is necessarily 
small since we have assumed $|\gw|\le 2k$. Note the analogy with~\eqref{F-l-c}.

For $\gw$ valid of length at most $2k-1$, let us now prove the
following identity:
\begin{multline}\label{Uw-c}
\Uc{\gw}=  
 	\mathbbm{1}_{\gw=\vareps} + 2tx \Uc{\vareps}\Uc{\gw} + tx
        \sum_{\gw=\gu a \gv }\Uc{\gu} x^{|\gv|/2} \Tc_\gv
+ t x^{|\gw|/2} \sum_{\gw=a\gu \bar a \gv} \Tc_\gu \Tc_\gv
\\
+ t\sum\limits_{\substack{\gu=\gv\gw 
 \\2\le  |\gu| \le 2k\\ \gu \text{ balanced } } } x^{|\gu|/2}
\Ucun{\gu_s}
+\frac{tx}{x-1} \left( \Uc{\gw}-x^k\Ucun{\gw}\right)-\frac{tx}{x-1} \sum_{\substack{\gu=\gv\gw\\ |\gu| \le 2k-2}}
 \Tc_\gu (x^{|\gu|/2}-x^k).
\end{multline}
The first line is similar to the first line of~\eqref{L-l-c2}: it  counts
the atomic map and orientations obtained from a merge. The only
difference is that we now record the root degree. On the second
line, the first sum counts orientations obtained by a small split
(with root word  $\gu$).
Let us explain the remaining terms, which count orientations obtained
by a large split, legal or not, of an orientation $O'$ whose root word
ends (necessarily) with $\gw$. Given an orientation $O'$ with  root vertex
degree $2d$, with $d>k$, the \gf\ of orientations obtained from $O'$ by a large
split is
$$
t^{1+\ee(O')} \left( x^{k+1}+ x^{k+2} + \cdots + x^d\right) =
t^{1+\ee(O')}\ \frac{x^{d+1}- x^{k+1}}{x-1}.
$$
Let us underline that we cannot apply a large split  to an orientation $O'$
whose root word $\gu$ satisfies $|\gu|\le 2k$. Hence the \gf\ of
orientations obtained by a large split is
$$
\frac {tx} {x-1}\left( \left( \Uc{\gw}- \sum_{\gu=\gv\gw, |\gu| \le 2k}
  x^{|\gu|/2} \Tc_\gu\right) - x^k \left( \Ucun{\gw}- 
\sum_{\gu=\gv\gw, |\gu| \le 2k}
 \Tc_\gu\right)\right),
$$
which gives the last two terms of~\eqref{Uw-c} (the terms
$\Tc_\gu$ with $|\gu|=2k$ do not contribute).

\medskip
\begin{remark}
  In the proof of~\eqref{Uw-c}, we have tried to follow the same steps
  as in the proof of~\eqref{L-l-c2}.  However, comparing~\eqref{Tw-c}
  and~\eqref{Uw-c} suggests to replace~\eqref{Uw-c} by a lighter
  equation:
\begin{multline}\label{Uw-c-bis}
\Uc{\gw}=  
  x^{|\gw|/2}\Tc_\gw + 2tx \Uc{\vareps}\Uc{\gw} + tx
        \sum_{\gw=\gu a \gv }\Uc{\gu} x^{|\gv|/2} \Tc_\gv
\\
+ t\sum\limits_{\substack{\gu=\gv\gw 
 \\ |\gu| \le 2k-1\\ \gu \text{ quasi-balanced } } } x^{(1+|\gu|)/2}
\Ucun{\gu}
+\frac{tx}{x-1} \left( \Uc{\gw}-x^k\Ucun{\gw}\right)-\frac{tx}{x-1}
\sum_{\substack{\gu=\gv\gw\\ |\gu| \le 2k-2}} 
 \Tc_\gu (x^{|\gu|/2}-x^k).
\end{multline}
\end{remark}

\begin{prop}\label{def-syst-u-c}
  Consider the collection of equations consisting of:
  \begin{itemize}
  \item Equation~\eqref{Tw-c}, written for all balanced words $\gw$
    of length between $2$ and $2k-2$,
\item Equation~\eqref{Uw-c-bis}, written for all  valid words $\gw$
of  length at most $ 2k-1$.
  \end{itemize}
In this collection, replace all trivial $\Tc$-series by their value:
$\Tc_\gw=0$ when $\gw$ is not balanced, $\Tc_\vareps=1$. Let $R_0$
denote the resulting system. The number of
series it involves is $\e(k)-{2k\choose k}$, where $\e(k)$ is given
by~\eqref{ek-def}. Moreover,  $R_0$ defines uniquely these series. Its
size  can be (roughly) divided by two upon exploiting the
$0/1$ symmetry.
\end{prop}
The proof is similar to the proof of Proposition~\ref{prop:l-c}.

\begin{remark}
  As in Remark~\ref{forward}, if $\gw$ is such that $0\gw$ and $1\gw$ are
both valid of length less than $2k$, we can replace~\eqref{Uw-c-bis} by the simpler forward equation:
$$
\Uc{\gw}= x^{|\gw|/2}\Tc_\gw + \Uc{0\gw}+ \Uc{1\gw}.
$$
This does not increase the size of the system.
\end{remark}
\subsection{Examples}

\subsubsection{When  $\boldsymbol{k=1}$,}
 the system of
Proposition~\ref{def-syst-u-c} contains $\e(1)-2=3$ equations. Upon
exploiting the 0/1 symmetry, it reads:
\beq\label{u-ck=1}
\begin{cases}%
\Uc{\vareps}\! \!\!\!& = \ 1+2tx(\Uc{\vareps}) ^2 + 2tx \Ucun{0} + \frac{tx}{x-1}
(\Uc{\vareps} -x \Ucun{\vareps}) +tx,
\\
\Uc{0}\!\!\!\!& = \  2tx\Uc{\vareps}\Uc{0} + tx \Uc{\vareps}+ tx \Ucun{0} + \frac{tx}{x-1}
(\Uc{0} -x \Ucun{0}).
\end{cases}
\eeq
The first equation can be replaced by the forward equation
$\Uc{\vareps}=1+2\Uc{0}$. 
We explain in Section~\ref{sec:quadratic} how to solve this system.

\subsubsection{When  $\boldsymbol{k=2}$,}
 the system of
Proposition~\ref{def-syst-u-c} contains $\e(2)-6=15$ equations. Upon
exploiting  the 0/1
symmetry and (some) forward equations, it
reads:
\beq\label{u-c-k=2-gros}
\begin{cases}%
\Tc_{10} \!\!\!\!& = \  t +t \Ucun{0},
\\
\Uc{\vareps}\!\!\!\!& = \ 1+2\Uc{0},
\\
\Uc{0}=\Uc{1} 
\!\!\!\!& = \  2tx\Uc{\vareps}\Uc{0}+tx\Uc{\vareps}+ t x\Ucun{0}
+tx^2(\Ucun{100}+\Ucun{010}+\Ucun{110})
+ \frac{tx}{x-1} ( \Uc{0}-x^2\Ucun{0}) +tx^2\Tc_{10},
\\
 \Uc{10}=\Uc{01}\!\!\!\!& = \  x\Tc_{10}+\Uc{110}+\Uc{010},
\\
\!\!\!\!& = \  x \Tc_{10}+2tx \Uc{\vareps}\Uc{10} +
tx\Uc{1}+tx^2(\Ucun{010}+\Ucun{110})
+ \frac{tx}{x-1}(
\Uc{10}-x^2\Ucun{10}) +tx^2\Tc_{10},
\\
\Uc{00}=\Uc{11}\!\!\!\!& = \  2tx \Uc{\vareps} \Uc{00} + tx
\Uc{0}+tx^2\Ucun{100} + \frac{tx}{x-1}( \Uc{00}-x^2\Ucun{00}),
\\
\Uc{100}\!\!\!\!& = \  2tx \Uc{\vareps} \Uc{100} + tx\Uc{10}+tx^2\Ucun{100} + \frac{tx}{x-1}( \Uc{100}-x^2\Ucun{100}),
\\
\Uc{010}\!\!\!\!& = \  2tx \Uc{\vareps} \Uc{010} + tx
\left(\Uc{01}+\Uc{\vareps} x T_{10}\right) 
+tx^2\Ucun{010}
+ \frac{tx}{x-1}( \Uc{010}-x^2\Ucun{010}),
\\
\Uc{110}\!\!\!\!& = \  2tx \Uc{\vareps} \Uc{110} + tx
\left(\Uc{11}+\Uc{\vareps} x T_{10}\right) 
+tx^2\Ucun{110}
+ \frac{tx}{x-1}( \Uc{110}-x^2\Ucun{110}).
\end{cases}
\eeq
We explain in Section~\ref{sec:quadratic} below how to solve this
system.

\subsection{Algebraicity}
\label{sec:alg-c}
Since the early work of Brown in the sixties on the \emm
quadratic method,~\cite{brown-square}, a lot has been known about equations involving divided differences of the
form $(F(t;x)-F(t;1))/(x-1))$. However, most of the literature deals with a single
equation, not with a system~\cite{mbm-jehanne,goulden-jackson}. In order to prove that the series
$\Ucl_\vareps^{(k)}(t;x)$ that counts 
orientations of $\peop$ is algebraic, we use a deep theorem from
Artin's approximation theory, due to Popescu~\cite{kurke,popescu}. The
form we will 
need is given below. We recall
  that $\cs[[z_1, \ldots, z_\ell]]$ is the ring of \fps\ in the
  variables $z_1, \ldots, z_\ell$, with complex coefficients, and that a series $Z$ in this ring is
  \emm algebraic, if it satisfies a non-trivial polynomial equation $\Pol(z_1,
  \ldots, z_\ell,Z)=0$.

\begin{theorem}[\cite{popescu}, Thm.~1.4]\label{thm:popescu}
Consider a polynomial system of $n$ equations in  $\ell+n$ variables over $\cs$, written
as $P_i(z_1, \ldots, z_\ell,y_1, \ldots, y_n)=0$, for $1\le i \le
n$. Let  $(d_1,
\ldots, d_n)$ be a sequence of integers in $\{0, 1, \ldots, \ell \}$.
Assume that there exists an $n$-tuple $\mathcal Y=(Y_1, \ldots, Y_n)$ of 
series in $\cs[[z_1, \ldots, z_\ell]]$ that satisfies the following
conditions:
\begin{itemize}
\item the $n$-tuple $\mathcal Y$ solves this system, that is,
$$
P_i(z_1, \ldots, z_\ell,Y_1, \ldots, Y_n)=0 \qquad \hbox{for }
1\le i\le n,
$$
\item for
 $1\le i\le n$, the series $Y_i$ does not
depend on the variables $z_j$ such that $j>d_i$
 (if
  $d_i=\ell$, then there is no condition on the series $Y_i$).
\end{itemize}
Then there exists an $n$-tuple $(Z_1, \ldots, Z_n)$ of \emm algebraic,
series in  $\cs[[z_1, \ldots, z_\ell]]$ that solves the system and
satisfies the same dependence conditions as $\mathcal Y$.

In particular, if the system has a unique solution satisfying the
dependence conditions, then this solution is algebraic.
\end{theorem}

\noindent{\bf An application.} To our knowledge, this theorem has not been applied yet in a
combinatorial context. So, before we use it to prove the algebraicity
of $\Ucl^{(k)}_\vareps(t;x)$, let us examine its application to a simple
equation, namely~\eqref{eq-func-M}.

First, observe that the algebraicity of $M(t;x)$ is not
obvious. Clearly, if we could prove that $M(t;1)$ is algebraic, we
would be done with $M(t;x)$ as well, but why should $M(t;1)$ be
algebraic? We can apply the above theorem as follows. Let us denote
$t=z_1$ and $x=1+z_2$ (we shall explain later why we need to
translate the  variable $x$). We consider the system in
$z_1, z_2$, $y_1$ and $y_2$ consisting of the following (single) equation:
$$
z_2 y_2= z_2 + z_1 z_2 (1+z_2) y_2^2 + z_1(1+z_2) (y_2-y_1).
$$
Take $d_1=1$ and $d_2=2$. Then~\eqref{eq-func-M} shows that the pair
$(Y_1, Y_2):=(M(t;1), M(t;x))$ solves the above equation. Moreover
$Y_1=M(t;1)$ is independent of $z_2=x-1$, while $Y_2=M(t;x)$
depends on both variables $z_1$ and $z_2$, in accordance with
$d_1=1$ and $d_2=2$. 

Let us now prove that there cannot be another solution of this system
in the ring $\cs[[z_1,z_2]]$
such that $Y_1$ is independent of $z_2$. First, setting $z_2=0$ in
the equation shows that $Y_1$ must be the specialization of $Y_2$ at
$z_2=0$. This, combined with the factor $z_1$ occurring in every
non-initial term in the right-hand side, implies that
the coefficient of $z_1^n$ in $Y_2$ can be computed by induction of
$n$, starting from the constant coefficient $1$. Hence the uniqueness of
$(Y_1,Y_2)$. The algebraicity of $M(t;x)$ now follows from the above theorem.

Note that, if we had used $z_2=x$ instead of $z_2=x-1$, we could not
apply the last part of Theorem~\ref{thm:popescu}. The equation would read 
$$
(z_2-1) y_2= (z_2-1) + z_1 (z_2-1) z_2 y_2^2 + z_1z_2 (y_2-y_1),
$$
but this equation has many solutions in the ring $\cs[[z_1,z_2]]$ of
\fps\ in $z_1=t$ and $z_2=x$. For instance, one can take $Y_1=0$ and  
$$
Y_2= \frac{1-x+tx -\sqrt{(1-x+tx)^2-4tx(1-x)^2}}{2tx(1-x)}.
$$
Theorem~\ref{thm:popescu} tells us that at least one of these
solutions is algebraic, but we need uniqueness to conclude that \emm our,
solution is algebraic. The key point is that a series in $\cs[[z_1, z_2]]$ can always be
specialized at $z_2=0$, but not at $z_2=1$. 

\medskip
We now apply Theorem~\ref{thm:popescu} to the larger example of
orientations of $\peop$.

\begin{prop}\label{prop:alg-c}
  For any $k\ge 1$, the \gf\ $\Ucl_\vareps^{(k)}(t;x)$ that counts orientations
  of $\peop$ is algebraic.
\end{prop}
\begin{proof}
  Again, we take as variables $z_1=t$ and $z_2=x-1$. For short, we
  denote $z_2$ by $z$.  We consider
  the  polynomial system consisting of the following equations, which
  mimic~\eqref{Tw-c} and~\eqref{Uw-c-bis}. For
  $\gw$ balanced of length between 2 and  $2k-2$,
$$
\Ac_\gw=
	t \sum\limits_{a \gu \bar{a} \gv = \gw} \Ac_{\gu} \Ac _{\gv}
	 +t  \Cc_{\gw_s} ,
$$
and for $\gw$ valid of length at most $2k-1$,
   \begin{multline}\label{eqU-popescu}
  z   \Bc_\gw =  
 z (1+z)^{|\gw|/2}\Ac_\gw + 2tz(1+z) \Bc_\vareps \Bc_\gw  + tz(1+z)
        \sum_{\gw=\gu a \gv }\Bc_\gu  (1+z)^{|\gv|/2} \Ac_\gv
\\+ tz\sum\limits_{\substack{\gu=\gv\gw 
 \\ |\gu| \le 2k-1\\ \gu \text{ quasi-balanced } } } (1+z)^{(1+|\gu|)/2}
\Cc_\gu
\\
+{t(1+z)} \left( \Bc_\gw -(1+z)^k\Cc_\gw \right)-{t(1+z)} \sum_{\substack{\gu=\gv\gw\\ |\gu| \le 2k-2}}
 \Ac _\gu ((1+z)^{|\gu|/2}-(1+z)^k),
   \end{multline}
where $\Ac_\vareps=1$. The variables $\Ac_\gw$, $\Bc_\gw$ and $\Cc_\gw $
play the role of the  $y_i$ in Theorem~\ref{thm:popescu}.  By construction, the series 
$$
\Ac_\gw:= \Tc_\gw(t), \qquad \Bc_\gw:= \Ucl_\gw(t;1+z), \quad \Cc_\gw:=
\Ucl_\gw(t;1)
$$
solve the system. Moreover,  $\Ac_\gw$ and
$\Cc_\gw$ do not depend on $z_2=z$.

By Theorem~\ref{thm:popescu}, it suffices to prove that the system in $\Ac, \Bc, \Cc$ has a
unique solution in $\cs[[t,z]]$ satisfying these dependence relations to conclude that
all our series $\Tc$ and 
$\Ucl$ counting orientations are algebraic. 

So assume that  $\Ac_\gw$, $\Bc_\gw$  and
$\Cc_\gw$ solve the system and satisfy
the required dependences. Then by setting $z=0$
in~\eqref{eqU-popescu}, we see that $\Cc_\gw$ must be the
specialization of $\Bc_\gw$ at $z=0$, for all $\gw$ valid of length
at most $2k-1$. Then the form of the system implies that the
coefficient of $t^n$ in all series can be computed by induction on
$n$, the initial values being $\Bc_\vareps=\Cc_\vareps=1+O(t)$ and
$\Ac_\gw=\Bc_\gw=\Cc_\gw=O(t)$ for $\gw$ non-empty (recall that we
have set $\Ac_\vareps=1$). This proves the uniqueness of the solution
(with the required dependences) and concludes the proof.
\end{proof}

\subsection{Back to examples}
\label{sec:quadratic}
\subsubsection{The case  $\boldsymbol{k=1}$.}
\label{subsec:quadratic1}
Let us go back to  System~\eqref{u-ck=1}. In the first
  equation, replace  $\Ucun{0}$ by $(\Ucun{\vareps}-1)/2$
to obtain a
single equation involving only  $\Uc{\vareps}=\Ucl_\vareps(x)$ and
$\Ucun{\vareps}=\Ucl_\vareps(1)$. For simplicity, we now drop the index
$\varepsilon$. This equation reads:
$$
\Pol(\Ucl(x), \Ucl(1), t, x)=0,
$$
with
$$
\Pol(x_0, x_1, t, x)= 
(x-1)(-x_0+1+2txx_0 ^2 + tx (x_1-1)) + {tx}
(x_0 -x x_1) +tx(x-1).
$$
We apply Brown's \emm quadratic method,. Its principle is the
  following: if there exists a \fps\ $X\equiv X(t)$ such that 
\beq\label{eq-der}
\Pol_{x_0}( U(X), U(1),t,X)=0 ,
\eeq
then this series $X$ must be a double root of the discriminant
$\Delta(\Ucl(1),t,x)$ of
$\Pol(x_0, \Ucl(1), t,x)$ with respect to $x_0$ (the notation
$\Pol_{x_0}$ stands for the derivative of $\Pol$ with respect to its
first variable). The proof is elementary
(see~\cite[Sec.~2.9]{goulden-jackson} or \cite{mbm-jehanne}). Equation~\eqref{eq-der}
reads 
$$
X=1+tX+4tX(X-1)\Ucl(X),
$$
and has a unique power series solution $X(t)$, whose coefficients can
be computed by induction from those of $\Ucl(x)$ (we do not need to
determine $X$, just to know that it exists). Thus $X$ is a double root
of $\Delta(\Ucl(1),t,x)$, and hence the discriminant in $x$ of $\Delta$
must vanish. This gives the following cubic
equation for $\Ucl(1)$ (see our {\sc Maple} sessions): 
\beq\label{cubic}
64 t^3 \Ucl(1)^3+2 t (24 t^2-36 t+1) \Ucl(1)^2+(-15 t^3+9 t^2+19 t-1)
\Ucl(1)+t^3+27 t^2-19 t+1=0.
\eeq
The series $\Ucl(1)$ has a unique positive singularity $\tau_1$, around
$0.0765$, which is a root of $216t^3-81t^2+18t-1$.  This gives the upper bound $\mu_1=1/\tau_1=13.0659\cdots$ on the
growth rate of Eulerian orientations. Expanding the series near
$\tau_1$ (using for instance the {\sc Maple} function {\tt algeqtoseries}~\cite{gfun}) shows
that it has a  singularity in $(1-\mu_1 t)^{3/2}$, as the \gf\ of
many families of planar maps.

\subsubsection{The case  $\boldsymbol{k=2}$.}
\label{sec:sol-u-k2-c}
 We now return to the
system~\eqref{u-c-k=2-gros}.  Observe that we can reduce it to a system
of three equations defining the series $\Uc{\vareps}, \Uc{10}$ and
$\Uc{100}$: 
\beq
\label{u-c-k=2}
\begin{cases}%
\Uc{\vareps}\!\!\!\!& = \ 1+  4tx\Uc{\vareps}\Uc{0}+2tx\Uc{\vareps}+
2t x\Ucun{0}+2tx^2(\Ucun{100}+\Ucun{10}) + \frac{2tx}{x-1} (
\Uc{0}-x^2\Ucun{0}) ,
\\
\Uc{10}\!\!\!\!& = \   x t(1+\Ucun{0})+2tx \Uc{\vareps}\Uc{10} +
{tx\Uc{0}}+tx^2\Ucun{10}+ \frac{tx}{x-1}(
\Uc{10}-x^2\Ucun{10}),
\\
\Uc{100}\!\!\!\!& = \  2tx \Uc{\vareps} \Uc{100} + tx\Uc{10}+tx^2\Ucun{100} + \frac{tx}{x-1}( \Uc{100}-x^2\Ucun{100}),
\end{cases}
\eeq
in which we still need to plug
\beq\label{specialization}
\Ucun{0}= \frac{\Ucun{\vareps}-1}{2} \qquad \hbox{and} \qquad\Uc{0}= \frac{\Uc{\vareps}-1}{2}.
\eeq

To solve this system, we could develop a matrix analogue of the
quadratic method, where~\eqref{eq-der} would be replaced by the cancellation
of the Jacobian of the system. However, we prefer a step by step
approach here, among other reasons because our system is
  not generic (its  Jacobian has a multiple root).

From now on, we lighten notation by denoting $A=\Uc{\vareps}$, $A_1=\Ucun{\vareps}$,
$B=\Uc{10}$, $B_1=\Ucun{10}$, $C=\Uc{100}$ and $C_1=\Ucun{100}$.
We
 will determine three polynomial equations relating the one-variable
series $A_1, B_1$ and $C_1$, and then
eliminate $B_1$ and $C_1$ to obtain a polynomial
equation satisfied by $A_1=\Ucun{\vareps}$. 

We now describe the various steps  of our calculation, without giving the
intermediate equations: we refer to our web pages
 for
a Maple session where the calculations are performed. 

The first equation of~\eqref{u-c-k=2}, after
injecting~\eqref{specialization}, involves only one $x$-dependent
series, namely $A=\Uc{\vareps}=\Ucl_\vareps(x)$.  Once the denominators are cleared
out, the degree in $A$ is 2, and we can apply  the quadratic
method of Section~\ref{subsec:quadratic1}: the discriminant (in $x$) of a certain
discriminant (in $x_0$) vanishes, and this gives a first equation
between $A_1, B_1$ and $C_1$.

We then move to the second equation of~\eqref{u-c-k=2}, which (after
injecting~\eqref{specialization}) involves two
$x$-dependent series, namely $A$ and $B$. It
is linear in the latter series, with coefficient:
\beq\label{coeff1}
1-x+tx+2tx(x-1)A.
\eeq
This coefficient vanishes for a (unique) series in $t$, denoted
$X$, satisfying
$$
X=1+tX+2tX(X-1)A_2, \qquad \hbox{with } A_2:=\Ucl_\vareps(X).
$$
Replacing $x$ by $X$ in the second equation of~\eqref{u-c-k=2} gives
another equation between $X$ and 
$A_2$, from which we compute
\beq\label{algX}
2t(A_1-2B_1)X^2+(1-t-2tA_1)X=1,
\eeq
$$
A_2= \frac{ 2X B_1}{X-1}-A_1.
$$
We now eliminate $X$ and $A_2$ between  the last two identities and the first equation of~\eqref{u-c-k=2}, specialized
at $x=X$. This gives a second equation between our three main
unknown series $A_1, B_1$ and $C_1$.

We finally consider the third equation of~\eqref{u-c-k=2} (after
injecting~\eqref{specialization}), which now
involves all three $x$-dependent series. It is linear in $C$,
again with coefficient~\eqref{coeff1}. Setting $x=X$ in this equation
gives an expression of $\Ucl_{10}(X)$:
$$
B_2:=\Ucl_{10}(X)= X C_1/ (X-1).
$$
 We now get back to the
second equation of~\eqref{u-c-k=2}, differentiate it with respect to
$x$ and set $x=X$. Replacing $B_2$ and $A_2$  by the above
expressions gives:
$$
A_2':=\frac{\partial \Ucl_\vareps}{\partial x}(X)=
2\frac{2C_1t(2X-1)(X-1)A_1-4Xt(2X-1)B_1C_1-B_1t(X-1)-(t-1)(X-1)C_1}
{t(X-1)^2(4C_1X+X-1)}.
$$
It remains to differentiate the first equation of~\eqref{u-c-k=2} with respect to $x$,
specialize it at $x=X$, and plug  the above values of
$A_2', B_2$ and $A_2$ to obtain one more equation between $A_1, B_1, C_1$
and $X$. Eliminating $X$ thanks to~\eqref{algX} gives 
our third and last equation between $A_1, B_1$ and $ C_1$.

From this system, we eliminate $B_1$
and $C_1$, and obtain an equation of degree 27 for
$A_1=\Ucun{\vareps}$. Its dominant
coefficient does not vanish away from 0, and its discriminant has three
roots in $[1/10, 1/16]$ (where we know that the radius must be found),
respectively located 
around 0.07509, $0.07658$ and 0.07727. Following numerically the
 branches that start from $1$ at $t=0$ shows that the radius of $\Ucun{\vareps}$ is the
second one, giving the upper bound $\mu_2=13.057\ldots$ on
the growth rate $\mu$ of Eulerian orientations.  From numerical
estimates of the singular exponent, we predict that the series has
again a ``planar map'' singularity in $(1-\mu_2t)^{3/2}$. This is
known to hold for many series satisfying an equation with divided
differences~\cite{drmota-noy-universal}. This leads us to complete
Proposition~\ref{prop:alg-c} as follows.

\begin{conjecture}\label{conj:sup-c}
  For every $k$, the  algebraic series $\Ucl_\vareps^{(k)}(t;1)$ that counts orientations of
  $\peop$ has a unique dominant singularity $\tau_k=1/\mu_k$ which is of the
  planar map type: as $t$ approaches $\tau_k$ from below,
$$
\Ucl_\vareps^{(k)}(t;1)=c_0+c_1 (1- \mu_k t) + c_2 (1-\mu_kt)^{3/2}
\big(1+o(1)\big)
$$
with $c_2 \not = 0$.
\end{conjecture}

\section{Supersets of Eulerian orientations, via the  prime decomposition}
\label{sec:prime-upper}

In this section, we combine the illegal large splits of the
previous section with the prime decomposition of Section~\ref{sec:prime-dec} to obtain a new family of
supersets of Eulerian orientations. These new supersets $\psup$
satisfy $\psup\subset \peop$ (Proposition~\ref{prop:incl-u}), hence they give
better bounds on the growth rate $\mu$ than  those obtained from the standard decomposition. Many arguments are
similar to those of the previous section, and we give fewer details.

\medskip

Recall  from Section~\ref{sec:prime-dec} that an Eulerian orientation
is a sequence of {prime} Eulerian orientations, and that a prime (Eulerian) orientation 
 can be  obtained recursively from the atomic map by either:
\begin{itemize} 
\item adding a loop,  oriented in either way, around an orientation $O_1$,
\item or a legal split on a prime orientation $O' \in \PEO$, followed by
  the concatenation of an arbitrary Eulerian orientation $O''$ at the new
  vertex created by the split (Figure~\ref{fig:rec-prime}).
\end{itemize}

\begin{definition} \label{def:ub-prime}
Let $k\ge 1$. Let $\psup$ be the set of planar orientations obtained
recursively from the atomic map by either:
\begin{itemize} 
\item  concatenating a sequence of prime orientations of $\psup$,
\item or adding a loop,  oriented in either way, around an orientation
  $O_1$ of $\psup$, 
\item or performing a legal $i$-split on a prime orientation $O' \in
  \psup$, with $i\le k$, followed by
  the concatenation of an arbitrary  orientation $O''$ of $\psup$ at the new
  vertex created by the split (small split),
 \item or performing an arbitrary $i$-split on a prime orientation $O' \in
  \psup$, with $i> k$, followed by
  the concatenation of an arbitrary  orientation  $O''$of $\psup$ at the new
  vertex created by the split (large split). If the split is legal,
  then the new edge is given the only orientation that makes the
 root word balanced, otherwise the root edge is
  oriented away from the root vertex.
 \end{itemize}
\end{definition}

Again, the sets $\psup$ decrease to the set $\PEO$ of all Eulerian
orientations as $k$ increases, hence their growth rates $\bar \mu_k$
form a non-increasing sequence of {upper} bounds on $\mu$. We do not
know if this sequence converges to $\mu$. At any rate, the
  convergence appears to be rather slow, as shown by the estimates of
  $\bar \mu_k$ in Table~\ref{tab:general}.

\begin{prop}\label{prop:incl-u}
For $k\ge 1$, the superset of orientations $\psup$ is  contained in
the superset $\peop$ defined in
Section~\rm{\ref{sec:classic-upper}}.
\end{prop}
\begin{proof}
  We prove this by induction on the number of edges. The inclusion is
  obvious for orientations with no edges. Now let $O\in \psup$, having
  at least one edge. 

If $O$ is prime and is obtained by adding a loop around 
a smaller orientation $O_1$ of $\psup$  (second construction in
Definition~\ref{def:ub-prime}), then $O_1$ belongs to $\peop$ by the induction
hypothesis, and so does $O$ (first construction in
Definition~\ref{def:ub-classic}). 

Assume now that $O$ is prime and is obtained by an $i$-split in a prime
orientation $O'$ of $\psup$, followed by the concatenation of an
orientation $O''$ of
$\psup$ at the new vertex (third or fourth construction in
Definition~\ref{def:ub-prime}). Then the orientation $\tilde O$ obtained by concatenating
$O'$ and $O''$ at their root belongs to $\psup$ (first construction in
$\psup$) and hence to $\peop$
by the induction hypothesis. But then one can recover $O$ by
performing an $i$-split in $\tilde O$, which is allowed in $\tilde O$ as it was
allowed in $O'$. This is the second construction in
Definition~\ref{def:ub-classic}, hence $O$ is in $\peop$.

Assume finally that $O$ is obtained by concatenating a prime
orientation $P$ of $\psup$ and another orientation $O_2$ of $\psup$ (first construction in
Definition~\ref{def:ub-prime}). By the induction hypothesis, both $P$
and $O_2$ are in $\peop$. If the root edge of $P$ is a loop, deleting
it from $P$ leaves an orientation $O_1$ which is in $\peop$. Then we can reconstruct $O$
by a merge of $O_1$ and $O_2$ as in the first construction of
Definition~\ref{def:ub-classic}. If the root edge of $P$ is not a
loop (Figure~\ref{fig:prop_20}), then $P$ was obtained by the third or fourth construction in
Definition~\ref{def:ub-prime}: allowed split in a prime orientation $P'$ of
$\psup$, followed by the concatenation of some $O''\in \psup$ at the new
vertex. Let $\tilde O$ be obtained by concatenating $O''$, $P'$ and $O_2$
(in counterclockwise order) at their roots. Then $\tilde O$ is in
$\psup$, but also in $\peop$ by the induction hypothesis. Then $O$ can be recovered
by a split in $\tilde O$, which is allowed in $\tilde O$ as it was allowed in
$P'$ (the split may have been small in $P'$ and become large in
$\tilde O$,
because of the orientation $O_2$, but the converse is not
possible). This is the second construction in
Definition~\ref{def:ub-classic}, hence $O$ is in $\peop$. 
\end{proof}

\begin{figure}[h]
\centering
\includegraphics[scale=0.8]{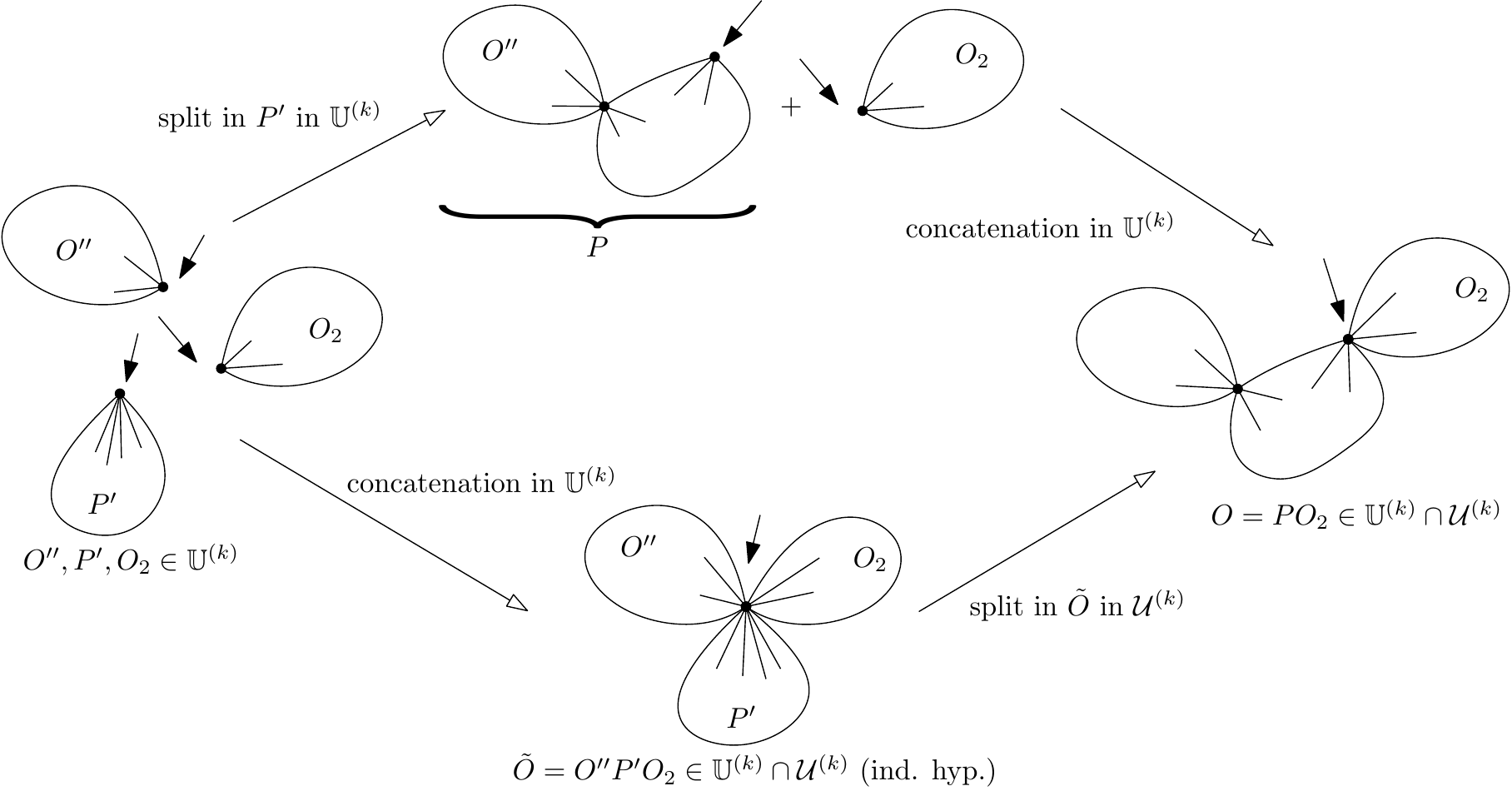}
\caption{Two constructions of the orientation $O$: (top) in $\psup$, via an $i$-split, (bottom) in $\peop$, via a $j$-split, with $j=i+\dv(O_2)$.}
\label{fig:prop_20}
\end{figure}

\subsection{Functional equations for $\boldsymbol\psup$}
We now fix an integer $k$.  For a word $\gw$ on $\{0,1\}$, let
$\Upl_\gw^{(k)}(t;x) \equiv \Upl_\gw(x)$ denote the generating function of
orientations of $\psup$ whose root word ends with $\gw$, counted by
the edge number (variable $t$) and the half-degree of the root vertex
(variable $x$). Let  $\Tp_\gw^{(k)}(t) \equiv \Tp_\gw$ denote the
generating function of orientations of $\psup$ having root word
exactly $\gw$. We define analogous \gfs\ $\Uppl_\gw(x)$ and $\Tpp_\gw$ for
prime orientations. As in the previous section, we often
  denote simply by $\Upun{\gw}$ (resp.~$\Uppun{\gw}$) the edge
  \gf\ $\Upl_\gw(t;1)$ (resp.~$\Uppl_\gw(t;1)$), and by $\Up{\gw}$ (resp.~$\Upp{\gw}$) the refined
  \gf\ $\Upl_\gw(t;x)$ (resp.~$\Uppl_\gw(t;x)$).

Note that $\Tp_\gw=\Tpp_\gw=0$ if $\gw$ is not balanced and that
$\Tp_\vareps=1$, $\Tpp_\vareps=0$.  For $\gw$ balanced of length between $2$ and
$2k$, we have both a sequential equation
\beq\label{Tp}
\Tp_\gw=\sum_{\gw=\gu\gv} \Tp_\gu \Tpp_\gv
\eeq
analogous to~\eqref{F-l-p}, and an equation for prime orientations:
\beq\label{Tpp}
\Tpp_\gw=t \Tp_{\gw_c} + t\Upun{\vareps} \Uppun{\gw_s},
   \eeq
analogous to~\eqref{Fp-l-p}. The factor $\Upun{\vareps}$ accounts for
the orientation concatenated after a split.

For $\gw$ valid of length at most $2k-1$, we have a sequential
equation, analogous to~\eqref{L-l-p} but taking care of the root degree:
\beq\label{Up}
\Up{\gw} = \mathbbm{1}_{\gw=\vareps} + \Up{\vareps} \Upp{\gw}+
\sum_{\gw= \gu\gv, \gv \not = \gw} \Up{\gu} x^{|\gv|/2}
\Tpp_\gv.
\eeq
Finally, we have the following equation for prime orientations, which
is the counterpart of~\eqref{Lpp-l-p} and involves ingredients
of~\eqref{Uw-c} for orientations obtained by a split:
\begin{multline}\label{Upp}
\Upp{\gw}=  
2tx \Up{\vareps}\mathbbm{1}_{\gw =\vareps}+ tx
\Up{\gw_p}\mathbbm{1}_{\gw\not = \vareps}+ tx^{|\gw|/2}
\Tp_{\gw_c} \mathbbm{1}_{\gw\not = \vareps \hbox{ balanced}}\\
+t \Upun{\vareps} \left(
\sum\limits_{\substack{\gu=\gv\gw 
 \\ 2\le|\gu| \le 2k\\ \gu \text{ balanced } }} x^{|\gu|/2}
\Uppun{\gu_s}\right.
\left. +\frac{x}{x-1} \left( \Upp{\gw}-x^k\Uppun{\gw}\right)-\frac{x}{x-1} \sum_{\substack{\gu=\gv\gw\\ |\gu| \le 2k-2}}
 \Tpp_\gu (x^{|\gu|/2}-x^k)\right).
\end{multline}
The first line counts orientations obtained by adding a loop, and the
second those obtained by a split.

\begin{remark}
  As in the previous section,  we can use~\eqref{Tpp} to
  replace~\eqref{Upp} by a   slightly lighter equation: 
\begin{multline}\label{Upp-light}
\Upp{\gw}=  x^{|\gw|/2}\Tpp_\gw+
2tx \Up{\vareps}\mathbbm{1}_{\gw =\vareps}+ tx \Up{\gw_p}\mathbbm{1}_{\gw \not=\vareps}
+t \Upun{\vareps} \left(
\sum\limits_{\substack{\gu=\gv\gw 
 \\ |\gu|\leq 2k-1\\ \gu \text{ quasi-balanced } } } x^{(1+|\gu|)/2}
\Uppun{\gu}
\right.\\ \left. 
+\frac{x}{x-1} \left( \Upp{\gw}-x^k\Uppun{\gw}\right)-\frac{x}{x-1} \sum_{\substack{\gu=\gv\gw\\|\gu| \le 2k-2}}
 \Tpp_\gu (x^{|\gu|/2}-x^k))\right).
\end{multline}
\noindent
\end{remark}

\begin{prop}\label{def-syst-u-p}
  Consider the collection of equations consisting of:
  \begin{itemize}
  \item Equation~\eqref{Tp}, written for all balanced words $\gw$
    of length between $2$ and $2k-4$,
 \item Equation~\eqref{Tpp}, written for all balanced words $\gw$
    of length between $2$ and $2k-2$,
\item Equation~\eqref{Up}, written for all  valid words $\gw$
of  length at most $ 2k-2$,
\item Equation~\eqref{Upp-light}, written for all  valid words $\gw$
of  length at most $ 2k-1$.
  \end{itemize}
In this collection, replace all trivial $\Tp$- and $\Tpp$-series by their value:
$\Tp_\gw=\Tpp_\gw=0$ when $\gw$ is not balanced, $\Tp_\vareps=1$,
$\Tpp_\vareps=0$. Let $\sf R_0$
denote the resulting system. The number of
series it involves is $2\e(k)-3{2k\choose k}- {2k-2\choose k-1}
\mathbbm{1}_{k>1 }$, where $\e(k)$ is given
by~\eqref{ek-def}. Moreover,  $\sf R_0$ defines uniquely all these series. Its
size can be (roughly) divided by two upon exploiting the
$0/1$ symmetry.
\end{prop}
The proof is similar to the proofs of  Propositions~\ref{prop:l-c}
and~\ref{def-syst-l-p}. 

\begin{remark}
   As always,  we can alternatively write forward equations:
$$
\Up{\gw}= x^{|\gw|/2} \Tp_\gw + \Up{0\gw}+ \Up{1\gw}, \qquad 
\Upp{\gw}=   x^{|\gw|/2}\Tpp_\gw + \Upp{0\gw}+ \Upp{1\gw}.
$$
\end{remark}
\subsection{Examples}

\subsubsection{When  $\boldsymbol{k=1}$,}
 the system of
Proposition~\ref{def-syst-u-p} contains $2\e(1)-3\cdot 2=4$ equations. Upon
exploiting the 0/1 symmetry, it reads:
\beq
\label{u-pk=1}
\begin{cases}%
\Up{\vareps}\!\!\!\!& = \ 1+\Up{\vareps} \Upp{\vareps},
\\
\Upp{\vareps}\!\!\!\!& = \ 2tx\Up{\vareps} + t \Upun{\vareps} \left( 2x
  \Uppun{0}
+ \frac{x}{x-1}(\Upp{\vareps} -x \Uppun{\vareps}) \right),
\\
\Upp{0}\!\!\!\!& = \  tx\Up{\vareps} + t \Upun{\vareps} \left( x
  \Uppun{0}
+ \frac{x}{x-1}(\Upp{0} -x \Uppun{0}) \right).
\end{cases}
\eeq
The second equation can be replaced by the forward equation
$\Upp{\vareps}=2\Upp{0}$.

We solve this system in Section~\ref{sec:sol-up}.

\subsubsection{When  $\boldsymbol{k=2}$,}
 the system of
Proposition~\ref{def-syst-u-p} contains $2\e(2)-3\cdot 6 -2=22$ equations. Upon
exploiting  the 0/1
symmetry and the forward equations, it reads: 
\beq
\label{u-p-k=2}
\begin{cases}%
\Tpp_{01} \!\!\!\!& = \  t +t \Upun{\vareps} \Uppun{1}, 
\\
\Up{\vareps} \!\!\!\!& = \  1+ \Up{\vareps} \Upp{\vareps},
\\
\Up{0}\!\!\!\!& = \ \Up{\vareps} \Upp{0},
\\
 \Up{10}= \Up{01}\!\!\!\!& = \  \Up{\vareps} \Upp{10},
\\
 \Up{00}= \Up{11}\!\!\!\!& = \  \Up{\vareps} \Upp{00} ,
\\
\Upp{\vareps}\!\!\!\!& = \ 2\Upp{0},
\\
   \Upp{0}= \Upp{1}
\!\!\!\!& = \  tx\Up{\vareps} +t \Upun{\vareps} \left(x \Uppun{0}+
x^2(\Uppun{100}+\Uppun{010}+\Uppun{110}) \right.
\left.
  + \frac{x}{x-1} ( \Upp{0}-x^2\Uppun{0}) +x^2\Tpp_{10}\right),
\\
\Upp{10}=\Upp{01} \!\!\!\!& = \  x\Tpp_{10}+\Upp{110}+\Upp{010},
\\
\!\!\!\!& = \  x \Tpp_{10}+tx \Up{1} + t\Upun{\vareps} \left( x^2\Uppun{010}+
  x^2\Uppun{110} + \frac{x}{x-1}(
\Upp{10}-x^2\Uppun{10}) + x^2\Tpp_{10}\right),
\\
 \Upp{00}\!\!\!\!& = \  tx\Up{0}+t\Upun{\vareps} \left(x^2\Uppun{100} + \frac{x}{x-1}( \Upp{00}-x^2\Uppun{00})\right),
\\
 \Upp{100}\!\!\!\!& = \  tx\Up{10}+t \Upun{\vareps} \left(
x^2\Uppun{100} + \frac{x}{x-1}( \Upp{100}-x^2\Uppun{100})\right),
\\
 \Upp{010}\!\!\!\!& = \   tx \Up{01} +t \Upun{\vareps} \left(
x^2\Uppun{010} + \frac{x}{x-1}( \Upp{010}-x^2\Uppun{010})\right),
\\
 \Upp{110}\!\!\!\!& = \  tx\Up{11} +t \Upun{\vareps} \left(
x^2\Uppun{110} + \frac{x}{x-1}( \Upp{110}-x^2\Uppun{110})\right).
\end{cases}
\eeq
We explain in Section~\ref{sec:sol-up} below how to solve this
system.

\subsection{Algebraicity}
\label{sec:alg-p}
The analogue of Proposition~\ref{prop:alg-c} holds for the supersets
obtained via the prime decomposition.
\begin{prop}\label{prop:alg-p}
  For any $k\ge 1$, the \gf\ $\Upl_\vareps^{(k)}(t;x)$ that counts orientations
  of $\psup$ is algebraic.
\end{prop}
\begin{proof}
  Again, the idea is to apply  Theorem~\ref{thm:popescu} to the system of
  Proposition~\ref{def-syst-u-p}, after writing $x=1+z$. The proof is roughly the same as that of
  Proposition~\ref{prop:alg-c}: we define a polynomial system
  involving two variables, $t$ and $z$, and {six} families of unknowns
  $\Ap_\gw$, $\App_\gw$, $\Bp_\gw$, $\Bpp_\gw$, $\Cp_\gw$, $\Cpp_\gw$. The equations they
  satisfy are those of Proposition~\ref{def-syst-u-p}, rewritten with 
$$
\Tp_\gw \rightarrow \Ap_\gw, \qquad \Tpp_\gw \rightarrow \App_\gw, \qquad \Up{\gw}
\rightarrow \Bp_\gw, \qquad \Upp{\gw} \rightarrow \Bpp_\gw, \qquad  \Upun{\gw}
\rightarrow \Cp_\gw, \qquad \Uppun{\gw} \rightarrow \Cpp_\gw.
$$
In fact, the only series $\Upun{\gw}$ occurring in our system is
$\Upun{\vareps}$, so that the polynomial system we construct 
involves $\Cp_\vareps$, but no other $\Cp$-series.
The prescribed dependences are that the $\Ap, \App, \Cp$ and $\Cpp$
series are independent of $z$.
For instance, when $k=1$ we  convert~\eqref{u-pk=1} into:
$$
\begin{cases}%
\Bp_\vareps\!\!\!\!& = \ 1+\Bp_\vareps \Bpp_\vareps,
\\
\Bpp_\vareps\!\!\!\!& = \ 2t(1+z)\Bp_\vareps + t \Cp_\vareps \left( 2(1+z)
 \Cpp_0
+ \frac{1+z}{z}(\Bpp_\vareps -(1+z) \Cpp_\vareps) \right),
\\
\Bpp_0\!\!\!\!& = \  t(1+z)\Bp_\vareps + t \Cp_\vareps \left( (1+z)
 \Cpp_0
+ \frac{1+z}{z}(\Bpp_0 -(1+z)\Cpp_0) \right).
\end{cases}
$$
However, this is not sufficient, because this system does  \emm not,
imply that $\Cp_\vareps$ is $\Bp_\vareps$ at $z=0$ (it \emm does, however
imply that $\Cpp_\vareps$ is $\Bpp_\vareps$ at $z=0$, and similarly
for $\Cpp_0$ and $\Bpp_0$). Hence, in order to apply Popescu's
theorem, we need to add to our collection of equations the case $x=1$,
$\gw=\vareps$ of~\eqref{Up}, namely
$\Cp_\vareps= 1+
\Cp_\vareps\Cpp_\vareps$. The rest of the argument mimics the proof of Proposition~\ref{prop:alg-c}.
\end{proof}

\subsection{Back to examples}
\label{sec:sol-up}
We first return to the system~\eqref{u-pk=1} obtained for $k=1$. In the second equation, replace 
$\Up{\vareps}$ by $1/(1-\Upp{\vareps})$,  $\Upun{\vareps}$ by
$1/(1-\Uppun{\vareps})$ and $\Uppun {0}$ by $\Uppun{\vareps}/2$. This gives a polynomial equation involving
only $\Upp{\vareps}$ and  $\Uppun{\vareps}$, which can be solved by
the quadratic method already used in Section~\ref{sec:quadratic}. This gives for
$\Uppun{\vareps}$ a cubic equation. Getting back to
$\Upun{\vareps}=1/(1-\Uppun{\vareps})$, we obtain for the \gf\
$\Upl_\vareps^{(1)}$ of orientations in $\mathbb U ^{(1)}$ the same cubic equation~\eqref{cubic} as
for orientations of $\mathcal U^{(1)}$. In fact, one can check that $\mathbb U ^{(1)}=\mathcal U^{(1)}$.
Of course, the upper bound on $\mu$ is  $\bar \mu_1=\mu_1= 13.0659\ldots$

\medskip
Let us now solve the system~\eqref{u-p-k=2} obtained for $k=2$. 
It can be reduced to a system of three equations defining the
series $\Uppl_\vareps, \Uppl_{10}$ and $\Uppl_{100}$: 
\beq\label{u-p-k=2-s}
\begin{cases}%
\Upp{\vareps}
\!\!\!\!& = \  2tx\Up{\vareps} +t \Upun{\vareps} \left(x \Uppun{\vareps}+
2x^2(\Uppun{100}+\Uppun{10}) 
  + \frac{x}{x-1} ( \Upp{\vareps}-x^2\Uppun{\vareps}) \right),
\\
\Upp{10} \!\!\!\!& = \  xt(1 + \Upun{\vareps} \Uppun{\vareps}/2)
+tx (\Up{\vareps}-1)/2 + t\Upun{\vareps} \left( x^2\Uppun{10} + \frac{x}{x-1}(
\Upp{10}-x^2\Uppun{10}) \right),
\\
 \Upp{100}\!\!\!\!& = \  tx\Up{10}+t \Upun{\vareps} \left(
x^2\Uppun{100} + \frac{x}{x-1}( \Upp{100}-x^2\Uppun{100})\right),
\end{cases}
\eeq
in which  we inject
\beq\label{specializations}
 \Upun{\vareps}= \frac 1 {1-\Uppun{\vareps}} , \qquad
 \Up{\vareps}= \frac 1 {1-\Upp{\vareps}} \qquad \hbox{and } \qquad 
\Up{10}= \frac{\Upp{10}}{1- \Upp{\vareps}}.
\eeq
%
 We lighten notation by  denoting $A=\Upp{\vareps}$, $A_1=\Uppun{\vareps}$,
$B=\Upp{10}$, $B_1=\Uppun{10}$, $C=\Upp{100}$ and
$C_1=\Uppun{100}$, and we follow the steps used in Section~\ref{sec:sol-u-k2-c} to solve
System~\eqref{u-c-k=2-gros}. Again, we refer to our web pages for the
corresponding Maple session. The intermediate steps are as
follows. We first apply the quadratic method to the first equation. We
then turn to the second one. The equation satisfied by $X$ is
$$
X=1+\frac{t}{1-t-A_1},
$$
(Note that it gives $X$ explicitly in terms of $A_1$, whereas we had a
quadratic equation~\eqref{algX} in the previous case.)
We then derive
$$
A_2:=\Uppl_\vareps(X)=1+ t\ \frac{1-A_1}{t-2B_1(1-A_1)}.
$$
We finally consider the third equation of~\eqref{u-p-k=2-s} (after
injecting~\eqref{specializations}), and derive:
$$
B_2:=\Uppl_{10}(X)= (1-A_2)C_1/t.
$$
Then it follows from the
second equation that:
$$
A_2':=\frac{\partial \Uppl_\vareps}{\partial x}(X)=
\frac{2(1-A_1)(1-t-A_1)^2((1-t-A_1)C_1+2(-1+A_1)B_1^2+tB_1)}{(t-2B_1(1-A_1
  ))^3}
.
$$

At the end, we obtain an equation of degree 28 for
$A_1=\Uppun{\vareps}$, and then for $\Upun{\vareps}$. Its dominant
coefficient does not vanish away from 0, and its discriminant has only one
root in $[1/10, 1/16]$ (where we know that the radius must be found),
around $0.0766$. This gives the upper bound $\bar \mu_2=13.047\ldots$ on
the growth rate $\mu$ of Eulerian orientations. From numerical
estimates of the singular exponent, we predict that the series has
again a ``planar map'' singularity in $(1-\bar \mu_2t)^{3/2}$. 

\medskip
For $k=3$, $4$ and $5$, we have generated our systems of
  equations and computed the first 100
  coefficients of $\Upl_\vareps^{(k)}(t;x)$. From this we get the estimates
  of the  growth rates $\bar \mu_k$ shown in Table~\ref{tab:values}. The
  singularity still appears to be in $(1-\bar \mu_kt)^{3/2}$.
We conjecture that this holds for any $k$.

\begin{conjecture}\label{conj:sup-p}
  For every $k$, the  algebraic series $\Upl_\vareps^{(k)}(t;1)$ that counts orientations of
  $\psup$ has a unique dominant singularity $\bar \tau_k=1/\bar \mu_k$ which is of the
  planar map type: as $t$ approaches $\bar \tau_k$ from below,
$$
\Upl_\vareps^{(k)}(t;1)=c_0+c_1 (1- \mu_k t) + c_2 (1-\bar \mu_kt)^{3/2}
\big(1+o(1)\big)
$$
for $c_2\not =0$.
\end{conjecture}

\section{Final comments}
\label{sec:final}
 As  mentioned at the end of the introduction, it might be
easier to count Eulerian orientation of 4-valent maps. In such
orientations, each vertex has two in-going and two outgoing edges, so
that counting Eulerian orientations means solving the so-called \emm
ice-model,  on (random)
4-valent maps~\cite[Chap.~8]{baxter}. The number of Eulerian orientations of
a 4-valent planar map is known to be a third of the number of proper 3-colourings
of its dual~\cite{welsh-tutte}. Thus counting these orientations is equivalent to
counting 3-coloured planar quadrangulations. A number of enumerative results are
already known about coloured maps. In particular, 3-coloured planar
maps, and 3-coloured planar triangulations, have algebraic
\gfs~\cite{bernardi-mbm-alg,BDFG02b,tutte-census-maps,tutte-chromatic-revisited}. More
generally, $q$-coloured planar maps and triangulations
have  differentially algebraic \gfs. So
this could be true for 3-coloured quadrangulations as well, and hence
for Eulerian orientations of 4-valent maps. Several results on this problem  appear in the physics
literature, but there does not seem to be an explicit exact solution
at the moment~\cite{kostov,zinn-justin-6V-random}.

\medskip
To finish, let us recall two questions left open by this paper (beyond
the enumeration of Eulerian orientations!).
\begin{itemize}
\item Do the growth rates $\bar \mu_k$ of orientations of $\psup$
  decrease to the growth rate $\mu$ of Eulerian orientations, or
  to a larger value?
\item In Sections~\ref{sec:asympt-c} and~\ref{sec:asympt-p}, we have used a general result about positive
  irreducible systems of polynomial equations to prove that the
  \gfs\ of our subsets of Eulerian orientations have a square root
  singularity. Could one define a notion of positive irreducible system of
  polynomial equations \emm with divided differences, whose solutions
  would systematically exhibit a singularity in $(1-\mu t)^{3/2}$?
  Hopefully this would apply to our supersets of orientations and
  prove Conjectures~\ref{conj:sup-c} and~\ref{conj:sup-p}. A first
  step in this direction, applicable to a single equation, is achieved in~\cite{drmota-noy-universal}.
\end{itemize}

\bigskip
\noindent {\bf Acknowledgements.} We are grateful to Alin Bostan,
Jean-Charles Faugère, Tony Guttmann and Bruno Salvy for their help with
the algebraic systems, various calculations and suggestions. MBM also
thanks the organizers of the workshop ``Approximation and
Combinatorics'' in 2015 in CIRM (Luminy, F) Herwig Hauser and Guillaume Rond,  for very
interesting discussions on the algebraic aspects of functional
equations arising in map enumeration.


\bibliographystyle{plain}
\bibliography{oe.bib}

 \end{document}